\documentclass{amsart}
\usepackage{amsfonts,amssymb,amsmath,enumerate,verbatim,mathtools,tikz,bm,mathrsfs,tikz-cd,hyperref}
\hypersetup{
    colorlinks=true,
    linkcolor=blue,
    filecolor=magenta,      
    urlcolor=cyan,
}
\usepackage[utf8]{inputenc}
\usepackage[english]{babel}
\usetikzlibrary{matrix}
\usetikzlibrary{arrows} 

\usepackage[all]{xy}
\SelectTips{cm}{}
\usepackage{stmaryrd}

\DeclareMathOperator{\cone}{cone}

\newcommand{\inca}{\hookrightarrow}

\DeclareMathOperator{\amp}{amp}

\DeclareMathOperator{\codepth}{codepth}\DeclareMathOperator{\depth}{depth}

\DeclareMathOperator{\hh}{HH}
\DeclareMathOperator{\h}{H}
\DeclareMathOperator{\bb}{B}
\newcommand{\n}{\mathfrak{n}}
\newcommand{\lb}{\llbracket}

\newcommand{\rb}{\rrbracket}

\newcommand{\CC}{\mathcal{C}}

\newcommand{\VV}{\mathcal{V}}

\newcommand{\EE}{\mathcal{E}}

\newcommand{\Z}{\mathbb{Z}}

\newcommand{\PS}{\mathbb{P}}

\newcommand{\A}{\mathcal{A}}

\newcommand{\T}{\mathsf{T}}
\newcommand{\Poly}{\mathcal{S}}
\newcommand{\D}{\mathsf{D}}
\newcommand{\N}{\mathbb{N}}
\newcommand{\ov}[1]{\overline{#1}}
\newcommand{\vp}{\varphi}

\newcommand{\kos}[2]{{#1}/\!\!/{#2}}

\newcommand{\con}{\subseteq}

\pgfdeclarelayer{bg}    
\pgfsetlayers{bg,main} 
\newcommand{\x}{{\bm{x}}}
\newcommand{\g}{{\bm{g}}}
\newcommand{\del}{\partial}
\newcommand{\f}{{\bm{f}}}

\newcommand{\y}{{\bm{y}}}

\newcommand{\0}{{\bf{0}}}
\newcommand{\m}{\mathfrak{m}}

\newcommand{\e}{\epsilon}

\DeclareMathOperator{\Ima}{Im}
\DeclareMathOperator{\Ker}{Ker}

\DeclareMathOperator{\coker}{coker}

\DeclareMathOperator{\pd}{pd}

\DeclareMathOperator{\z}{Z}

\DeclareMathOperator{\cx}{cx}

\DeclareMathOperator{\Spec}{Spec}
\DeclareMathOperator{\Proj}{Proj}

\DeclareMathOperator{\Hom}{Hom}
\DeclareMathOperator{\Ext}{Ext}

\DeclareMathOperator{\V}{\mathsf{V}}
\DeclareMathOperator{\Var}{V}

\newcommand{\gsupp}[1]{\mathsf{Supp}^+_{#1}}
\DeclareMathOperator{\Thick}{\mathsf{Thick}}
\DeclareMathOperator{\Tor}{Tor}
\DeclareMathOperator{\Kos}{Kos}

\newcommand{\shift}{{\mathsf{\Sigma}}}
\DeclareMathOperator{\RHom}{\mathsf{RHom}}

\newcommand{\ot}{\otimes^{\mathsf{L}}}
\newcommand{\xla}{\xleftarrow}
\newcommand{\xra}{\xrightarrow}

\newtheorem{theorem}{Theorem}[subsection]

\newtheorem{proposition}[theorem]{Proposition}

\newtheorem{lemma}[theorem]{Lemma}

\newtheorem{quest}[theorem]{Question}

\theoremstyle{definition}
\newtheorem{example}[theorem]{Example}
\newtheorem{definition}[theorem]{Definition}

\newtheorem{Construction}[theorem]{Construction}

\theoremstyle{remark}
\newtheorem{remark}[theorem]{Remark}

\newtheorem*{Not}{Notation}

\newtheorem{Notation}[theorem]{Notation}
\newtheorem{Discussion}[theorem]{Discussion}
\newtheorem{Problem}[theorem]{Problem}

\newtheorem{chunk}[theorem]{}

\numberwithin{equation}{section}

\theoremstyle{theorem}

\newcounter{intro}
\newtheorem{introthm}[intro]{Theorem}

\theoremstyle{definition}

\begin{document}

\title[Cohomological supports over derived complete intersections and local rings]{Cohomological supports over derived complete intersections and local rings}

\author[Josh Pollitz]{Josh Pollitz}
\address{Department of Mathematics,
University of Utah, Salt Lake City, UT 84112, U.S.A.}
\email{pollitz@math.utah.edu}

\date{\today}

\thanks{The author was partly supported through National Science Foundation  grants DMS 1103176, DMS 1840190, and DMS 2002173.}


\maketitle

\begin{abstract}
A theory of cohomological support  for  pairs of DG modules over a Koszul complex is investigated.  These  specialize to the support varieties of Avramov and Buchweitz defined over a complete intersection ring, as well as support varieties  over an exterior algebra. The main objects of study are certain DG modules over a polynomial ring; these   determine the aforementioned cohomological supports and are shown to encode   (co)homological information about pairs of DG modules over a Koszul complex.   The perspective in this article leads to new proofs  of well-known results
 for pairs of complexes over a complete intersection. Furthermore, these cohomological supports are used to define a support theory for pairs of objects in the derived category of an arbitrary commutative noetherian local ring. Finally,  we calculate several examples; one of which answers a question of D. Jorgensen in the negative. 
\end{abstract}

\section*{Introduction}

In this article we study the cohomological properties of a Koszul complex over a commutative noetherian ring and introduce a corresponding theory of cohomological support. This  simultaneously generalizes  results in  two well-studied  settings:   (1) arbitrary deformations of a commutative noetherian ring, and (2)   exterior algebras defined over a commutative noetherian ring. This theory provides a unified perspective that leads to more natural proofs. Furthermore, the calculations of certain geometric invariants constitute one of the main contributions of this paper. 

 The setup is the following. Let  $Q$ denote a commutative noetherian ring, $\f=f_1,\ldots, f_n$  an \emph{arbitrary} list of elements in $Q$ and let $E$ denote  the Koszul complex on $\f$ over $Q$. We regard $E$ as a DG $Q$-algebra in the usual way, and in doing so  we recover the cases above by secializing to when (1) $\f$ is a $Q$-regular sequence, and (2) each $f_i=0$, respectively. 
 In the introduction we restrict to the case that $Q$ is a regular local ring with residue field $k$, since this is when the strongest results hold.    In this setting we  refer to $E$ as a \emph{derived complete intersection}. See Remark \ref{remark} for a discussion of the terminology.

 For DG $E$-modules  $M$ and $N$ with finitely generated homology, we associate a Zariski closed subset $\VV_E(M,N)$ of $\PS_k^{n-1}$ (see Definition \ref{defsupportintro}), called the cohomological support of the pair $(M,N)$ whose dimension records the polynomial growth rate of  the minimal number of generators of  $\Ext_E^*(M,N)$. The latter value is called the complexity of $(M,N)$, denoted $\cx_E(M,N)$  (cf. \ref{c:cx} for a precise definition). 
In Theorem \ref{MAIN1}, we show these supports satisfy the following. 
\begin{introthm}\label{intro1}
For DG $E$-modules  $M,M',N, N'$ with finitely generated homology,
\[\VV_E(M,N)\cap \VV_E(M',N')=\VV_E(M,N')\cap \VV_E(M',N).\] 	
\end{introthm}
A consequence of Theorem \ref{intro1}  is
  the following  bound and symmetry of complexity over derived complete intersections: \emph{For  DG $E$-modules $M$ and $N$ with finitely generated homology,} \[\cx_E(M,N)=\cx_E(N,M)\leq n.\] This recovers the asymptotic theorems of Avramov and Buchweitz  \cite[Theorem II]{SV}, for local complete intersections, and Avramov and Iyengar \cite[5.3]{AI3}, for exterior algebras.
The proof  in \cite{SV} puts to use a theory of intermediate hypersurfaces which reduces the study of Ext-modules over a complete intersection to the study of Ext-modules  over certain hypersurface rings; the latter are well-understood due to the foundational work of Eisenbud in \cite{Eis} (see also \cite{Buch}). The proof in \cite{AI3} uses the Hopf-algebra structure of an exterior algebra.

Theorem \ref{intro1} follows a different route. Namely, we relate  $\VV_E(M,N)$ to  the cohomological support (c.f. Definition \ref{defsup}) of a certain DG module over $\Poly=Q[\chi_1,\ldots,\chi_n]$. This perspective builds on ideas from \cite{CD2} and works with chain level operators corresponding to the Hochschild cohomology of $E$ over $Q$;  these operators are also those introduced by Gulliksen  \cite{G} and  studied in various avatars by Avramov, Eisenbud, Gasharov, Mehta, Peeva, Sun and many others (see, for example,  \cite{VPD,AGP,AS,Eis,Mehta}).  
The regularity of $Q$ is used in a fundamental way to establish  isomorphisms for  the DG $\Poly$-modules which determine the supports under consideration in Theorem \ref{intro1}; the equality of supports is a direct consequence of these isomorphisms.

The perspective adopted  in this paper provides a framework for studying cohomology in  settings where one cannot   resort to the study of intermediate hypersurfaces or exploit a Hopf-algebra structure. For example, \cite{FM} identified a class of color commutative rings that exhibit behavior similar to that of a local complete intersection that do not enjoy enough intermediate hypersurfaces nor, \emph{a priori}, a Hopf-algebra structure.  In a recent collaboration of the present author with Ferraro and Moore \cite{FMP}, we prove an analog of Theorem \ref{intro1}  and arrive at symmetries in complexity for such color commutative rings by following the proof strategy of Theorem \ref{intro1} described above.

Aside from the asymptotic information recorded in these supports, the regularity of the sequence $\f$, and hence, the complete intersection property  is detected by these varieties.  Theorem \ref{cmsv} establishes the following. 
   \begin{introthm}\label{var}
Let $E=\Kos^Q(\f)$ be a derived complete intersection  and  set $R=Q/(\f)$. The following are equivalent:
\begin{enumerate}
\item $R$ is a complete intersection.
\item $\VV_E(R,k)=\emptyset$.
\item  $\VV_E(M,k)=\emptyset$ for some nonzero finitely generated $R$-module $M$.
\end{enumerate}
\end{introthm}  
The equivalence of (1) and (2) in Theorem \ref{var} was first proven in \cite{Pol} but an independent proof of the  more general   Theorem \ref{cmsv} is given in this article. 
Theorem \ref{var} has been applied in \cite{Pol}, and more recently in \cite{BEJ}, to establish a triangulated category characterization of locally complete intersection rings, answering a question of Dwyer, Greenlees, and Iyengar \cite{DGI}. One of the major points is that the structure of thick subcategories is also reflected in these varieties (cf. Remark \ref{remarkdgi}).

Finally, we use this theory of supports over a derived complete intersection to  define a support theory for pairs of complexes over an arbitrary  local ring $R$; we denote these by $\VV_R(M,N)$ for complexes of $R$-modules $M$ and $N$, see Definition \ref{localdef} for details.
 These recover the supports from \cite{SV} for pairs of finitely generated modules over a local complete intersection, and  in Theorem \ref{unstable} it is shown that these also specialize to the more general support varieties in \cite{Jor}. Working at the chain level with the DG $\Poly$-modules described above allows  us to answer a question from  \cite{Jor} in the negative (see Example \ref{counter}): \emph{If $\VV_R(M,N)=\emptyset$ for some finitely generated $R$-modules $M$ and $N$, is $R$ a complete intersection?}  Moreover, Theorem \ref{var}  can be interpreted as a corrected form of the question  from \cite{Jor}. 

 It is worth noting that Theorem \ref{var} says the supports $\VV_R(M,k)$ for a finitely generated $R$-module $M$, especially $\VV_R(R,k)$, are geometric  obstructions to $R$ itself being a  complete intersection. In general,  embedded deformations of $R$ cut down $\VV_R(R,k)$ by a hyperplane (see Proposition \ref{p:ed}); a complete intersection being a ring with a maximal number of embedded deformations and hence, an empty support (cf. Theorem \ref{var}). 
However,  the general structure of $\VV_R(R,k)$ remains mysterious when $R$ is not a complete intersection. The final  result we highlight  makes progress on gaining  insight on $\VV_R(R,k)$  by characterizing the possible closed subsets  it  can be when $R$
has  small codepth (see Section \ref{cod}). 
\begin{introthm}\label{clsv1}
 If $R$ is not a complete intersection and $\codepth R\leq 3$, then $\VV_R(R,k)=\mathbb{P}_k^{n-1}$ except when $R$ admits an embedded deformation. In the exceptional case,  $\VV_R(R,k)$ is a hyperplane in $\mathbb{P}_k^{n-1}$.
\end{introthm}

Determining the possible closed subsets that can be realized by  $\VV_R(R,k)$  remains an interesting problem for arbitrary codepth. In hopes of better understanding this problem, we propose the following question: \emph{Is $\VV_R(R,k)$ always a union of linear spaces?}
The examples in this article  and  \cite[Section 4]{BEJ}, as well as various calculations using Macaulay2 \cite{M2}, have yet to produce an example that is not a union of linear spaces. 

\section*{Acknowledgements}
This paper was partially completed at the University of Nebraska-Lincoln while I was finishing my thesis  under the supervision of Luchezar Avramov and Mark Walker. It is a pleasure to thank both of them for   innumerable conversations relating to this work. I would also like to thank Srikanth Iyengar and the referee for helpful comments on  this article.

\section{Differential Graded Homological Algebra} \label{secprelim}

\subsection{Conventions and Notation}

Fix a commutative noetherian ring $Q$. Let $A=\{A_i\}_{i\in \Z}$ denote a commutative DG $Q$-algebra. By \emph{commutative}, we mean that $A$ is commutative in the graded sense; namely,  $$ab=(-1)^{|a||b|}ba$$ for all $a$ and $b$ in $A$. 
   
 \begin{chunk}A map $\vp: M\to N$ between DG $A$-modules $M$ and $N$ is called a \emph{morphism of DG $A$-modules}  provided that  $\vp$ is a morphism of the underlying complexes of $Q$-modules such that   $\vp(am)=a\vp(m)$ for all $a\in A$ and $m\in M$.  We  use the notation  $\vp: M\xra{\simeq} N$ to mean that the morphism of DG $A$-modules  $\vp$ is a quasi-isomorphism.  
\end{chunk}

\begin{chunk} Let $M$ be a DG $A$-module. The differential of  $M$ is denoted by  $\del^M$. We use $|m|$ to denote the degree of an element of $M$, i.e., $|m|=d$ exactly when $m\in M_d$. 
For each $i\in \Z$,  $\shift^i M$ is the DG $A$-module given by
 $$(\shift^i M)_{n}:=M_{n-i}, \  \del^{\shift^i M}:=(-1)^i\del^{M}, \ 
\text{ and } a\cdot m:=(-1)^{|a|i}am.$$  
The \emph{boundaries} and \emph{cycles of $M$} are  \[\bb(M):=\{\Ima \del^M_{i+1}\}_{i\in \Z} \ \text{ and }  \ \z(M):=\{ \Ker\del^M_i\}_{i\in \Z}, \]respectively. 
The \emph{homology of $M$} is defined to be  $$\h(M):=\z(M)/\bb(M)=\{\h_i(M)\}_{i\in \Z}$$ which is a graded module over the graded $Q$-algebra $\h(A):=\{\h_i(A)\}_{i\in \Z}.$  We let $M^\natural$ denote the underlying graded $Q$-module obtained by forgetting the differential of $M$; note that  $A^\natural$ is a graded $Q$-algebra and $M^\natural$ is a graded $A^\natural$-module. 
\end{chunk}

\begin{chunk}\label{c:grade}
As Ext-modules have historically been graded cohomologically, we will  be consistent with this  convention. In particular, when working with DG modules whose homology is a graded Ext-module of interest we will  grade these objects  cohomologically. An effort has been made to make it clear when working in this setting. All other graded objects will be  graded homologically as indicated above. We point out that when $M=\{M^i\}_{i\in \Z}$ is cohomologically graded DG $A$-module,
$\shift^j M$ has $(\shift^jM)^i=M^{i+j}$.  \end{chunk}

\subsection{Semifree and Semiprojective DG modules}

Besides setting terminology, the  goal of this section is to  establish the ``moreover" statement from  Proposition \ref{gpsp} below in the generality stated there; this wil  be needed for proving some of the main results of the article (for example, Theorem \ref{MAIN1}).   This section may be skipped by the experts.

\begin{chunk}\label{semiproj}
A DG $A$-module $P$ is  \emph{semiprojective} if for every morphism of DG $A$-modules $\alpha: P\to N$ and each surjective quasi-isomorphism of DG $A$-modules $\gamma: M\to N$ there exists a unique up to homotopy morphism of DG $A$-module $\beta: P\to M$ such that $\alpha=\gamma\beta$. Equivalently, $P^\natural$ is a projective graded $A^\natural$-module and $\Hom_A(P,-)$ preserves quasi-isomorphisms. When $P$ is semiprojective the functor $P\otimes_A-$ is also exact and  preserves quasi-isomorphisms.  \end{chunk}
\begin{chunk}\label{sres}
  A \emph{semiprojective resolution} of a DG $A$-module  $M$  is a surjective quasi-isomorphism of DG $A$-modules $\e: P\to M$ where  $P$ is a semiprojective DG $A$-module.   Semiprojective resolutions exist and any two semiprojective resolutions of $M$ are unique up to  homotopy equivalence \cite[6.6]{FHT}.  \end{chunk}

\begin{chunk}\label{2of3}
Assume that  $$0\to L\xra{\alpha} M \xra{\beta} N\to 0$$ is an exact sequence of DG $A$-modules   such that $$0\to L^\natural\xra{\alpha^\natural} M^\natural \xra{\beta^\natural} N^\natural\to 0$$ is a split exact sequence of graded $A^\natural$-modules. By the Five Lemma, it follows  that if  two of the three DG $A$-modules are semiprojective, so is the third. 
\end{chunk}

  \begin{chunk}\label{semifree}
  A DG $A$-module $F$ is \emph{semifree} if there exists a chain of DG $A$-submodules of $F$  $$0=F(-1)\con F(0)\con F(1)\con \ldots$$  such that $\bigcup_i F(i)=F$ and  for each $i$ $$F(i+1)/F(i)\cong \coprod_{x\in X_i} \shift^{|x|} A$$ where $X_i$ is some graded set. Every semifree DG $A$-module is semiprojective (see \cite[6.10]{FHT}). 
  \end{chunk}

The following construction  is the standard one used to  construct a semiprojective (in fact, semifree) resolution of a DG A-module; see, for example,   \cite[2.2.6]{IFR}. 
\begin{Construction}\label{formal}
Let  $M$ be a  DG $A$-module. For each $m\in M$ we let $e_m$ be a graded variable of degree $|m|$. We define $C(m)$ to be the DG $A$-module with $$C(m)^\natural :=Ae_m^\natural\oplus Ae_{\del^M(m)}^\natural$$ where $$\del(ae_m+be_{\del^{M}(m)})=\del^A(a)m+(-1)^{|a|}ae_{\del^{M}(m)}+\del^A(b)e_{\del^M(m)}.$$ It is straightforward to see that  $C(m)$ is a contractible   semifree DG $A$-module. Moreover, we have a morphism of DG $A$-modules $\pi(m): C(m)\to M$  given by $$ae_m+be_{\del^{M}(m)}\mapsto am+b\del^M(m).$$ 

For any cycle $z\in \z(M)$, we have a morphism of DG $A$-modules $\tau(z): Ae_z\to M$ given by $ae_z\mapsto az.$ Define $$C^M:=\left(\coprod_{m\in M} C(m)\right)\coprod\left(\coprod_{z\in \z(M)} Ae_z\right)$$ and define   $$\pi^M:= \left(\sum_{m\in M} \pi(m)\right)+\left(\sum_{z\in \z(M)} \tau(z)\right): C(M)\to M.$$ Finally,  set $K^M:=\Ker \pi^M$ and so we have a short exact sequence \begin{equation}
\label{exact1} 0\to K^M\xra{\iota^M} C^M\xra{\pi^M}M\to 0
\end{equation}
\end{Construction}

\begin{lemma}
\label{formal1}Using the notation from Construction \ref{formal},  the following hold: 
\begin{enumerate}
\item $C^M$ is a semifree DG $A$-module. 
\item  There exists a surjective homotopy equivalence $C^M\to \coprod_{z\in \z(M)} \shift^{|z|} A$.
\item Applying $\bb(-),$ $\h(-)$ or $\z(-)$ to (\ref{exact1}) yield  exact sequences of graded $\z(A)$-modules. 
\end{enumerate}
In particular, if $\del^A=0$, then  $\z(C^M)$ and $\bb(C^M)$    are free DG $A$-modules.  
\end{lemma}
\begin{proof}
 It is clear that $C^M$ is a semifree DG $A$-module and that $\pi^M$  and $\z(\pi^M)$ are surjective. Consider the following commutative diagram of DG $\z(A)$-modules 
 \begin{center}
 \begin{tikzcd}
   & 0 \arrow{d} & 0\arrow{d} & 0\arrow{d} & \\ 
   0 \arrow{r}& \z(K^M) \arrow{d} \arrow["\z(\iota^M)"]{r}& \z(C^M) \arrow{d}  \arrow["\z(\pi^M)"]{r}& \z(M)\arrow{r}  \arrow{d} & 0 \\
      0\arrow{r} & K^M\arrow{d} \arrow["\iota^M"]{r}& C^M\arrow{d} \arrow["\pi^M"]{r}& M\arrow{r} \arrow{d}& 0 \\
      0 \arrow{r}& \shift\bb(K^M) \arrow["\shift\bb(\iota^M)"]{r} \arrow{d}  & \shift\bb(C^M)\arrow["\shift\bb(\pi^M)"]{r}  \arrow{d} & \shift\bb(M) \arrow{r}  \arrow{d} & 0 \\
         & 0 & 0 & 0 &  
 \end{tikzcd}
 \end{center}
Since $\z(-)$ is left exact, it follows that the top two rows of  the diagram are exact. Also, the columns are the canonical exact sequences. Hence, by the \emph{nine lemma} it follows that the last row is exact. A similar argument now yields that applying $\h(-)$ to (\ref{exact1}) yields an exact sequence.

 Since $C(m)$ is contractible for each $m\in M$  and  $Ae_z\cong\shift^{|z|}A$ for each $z\in \z(M)$, we conclude that there is   a surjective homotopy equivalence $C^M\to \coprod_{z\in \z(M)} \shift^{|z|} A$. 

Now assume that $\del^A=0$. For each $m\in M$, $$\bb(C(m))=\z(C(m))=Ae_{\del^M(m)}.$$ Thus, $$\z(C^M)=\left(\coprod_{m\in M} Ae_{\del^M{m}} \right)\coprod\left(\coprod_{z\in \z(M)} Ae_z\right)$$ and $$\bb(C^M)=\coprod_{m\in M} Ae_{\del^M{m}}.$$  Thus, $\z(C^M)$ and $\bb(C^M)$ are free DG $A$-modules. \end{proof}

\begin{lemma}\label{dgres}
Let $M$ be a DG $A$-module. There exists an exact sequence of DG $A$-modules $$\cdots \to F^n\to F^{n-1}\to \ldots \to F^1\to F^0\to M\to 0$$ such that each $F^i$ is a semifree free DG $A$-module that maps  onto $\coprod_{z\in \z(K^i)} \shift^{|z|}A$ via a surjective homotopy equivalence  where $K^i=\coker(F^{i+1}\to F^i).$ Moreover,  the induced sequences of graded  $\z(A)$-modules are exact:
\begin{enumerate}
\item$  \cdots \to \bb(F^1)\to \bb(F^0)\to \bb(M)\to 0$
\item$ \cdots \to \z(F^1)\to \z(F^0)\to \z(M)\to 0$
\item$  \cdots \to \h(F^1)\to \h(F^0)\to \h(M)\to 0$
\end{enumerate}
Furthermore, if $\del^A=0$, $F^i$ can be chosen such that $\z(F^i)$ and  $\bb(F^i)$ are free DG $A$-modules.  
\end{lemma}
\begin{proof}
Iteratively applying Lemma \ref{formal1} to obtain exact sequences $$0 \to K^i\to F^i\to M^i\to 0$$ where $M^0:=M$ and $M^i:=K^{i-1}$ gives us the desired result. 
\end{proof}

\begin{proposition}\label{gpsp}
Suppose that $\del^A=0$. Let $M$ be a DG $A$-module where two of  $\bb(M)$, $\z(M)$, $\h(M)$ have finite projective dimension when regarded as  graded $A$-modules. There exists an exact sequence of DG $A$-modules \begin{equation}0 \to F^n\to F^{n-1}\to \ldots \to F^1\to F^0\to M\to 0\label{resdg}\end{equation} where $F^n$ is semiprojective and $F^i$ is semifree for $0\leq i \leq n-1$.  Moreover, if $M$ is graded projective, then $M$ is semiprojective. 
\end{proposition}
\begin{proof}

Let  $$\cdots \to F^n\to F^{n-1}\to \ldots \to F^1\to F^0\to M\to 0$$ be the sequence obtained from Lemma \ref{dgres}. Set  $K^n:=\ker (F^n\to F^{n-1})$. 
The assumption 
 and the exact sequences of DG $A$-modules  $$0\to \bb(M)\to \z(M)\to \h(M)\to 0$$ imply that $\bb(M)$ and $\z(M)$ have finite projective dimension when regarded as graded $A$-modules. Since (1) and (2) are graded free resolutions of $\bb(M)$ and $\z(M)$ over $A$, respectively, it follows that    $\bb(K^n)$  and $\z(K^n)$ are graded projective DG $A$-modules for  $n\gg 0$. As $\del^A=0$, $\del^{\z(K^n)}=0$, and $\del^{\bb(K^n)}=0$ it follows that $$\h(\Hom_A(\z(K^n),-))\cong \Hom_A(\z(K^n), \h(-))$$ $$\h(\Hom_A(\bb(K^n),-))\cong \Hom_A(\bb(K^n), \h(-)).$$ Thus, $\z(K^n)$  and $\bb(K^n)$ are semiprojective DG $A$-modules. Finally, $\ref{2of3}$ and the graded split exact sequence $$0\to \z(K^n)\to K^n\to \shift \bb(K^n)\to 0$$ yield that $K^n$ is semiprojective. Induction on the length  of the exact sequence (\ref{resdg})  and again applying $\ref{2of3}$ establishes that if $M$ is graded projective then $M$ is semiprojective.  
\end{proof}

\subsection{The Derived Category of a DG Algebra}\label{thick}
\begin{chunk}
Let $\D({A})$ denote the derived category of ${A}$; recall that $\D(A)$, equipped with $\shift$, is a triangulated category. Define  $\D^f({A})$ to be the full subcategory of $\D({A})$ consisting of all $M$ such that $\h(M)$ is a finitely generated graded module over  $\h({A}).$ We use $\simeq$ to denote isomorphisms in $\D(A)$ and reserve $\cong$ for isomorphisms of DG $A$-modules. 
\end{chunk}

  \begin{chunk} 
  For a DG $A$-module $M$, define   $$\RHom_A(M,-):=\Hom_A(P,-)\text{ and }M\ot_A-=P\otimes_A-$$ where $P$ is a semiprojective resolution of $M$ over $A$. By \ref{semiproj} and \ref{sres}, $\RHom_A(M,-)$ and $M\ot_A-$  are well-defined exact endo-functors on $\D(A)$.  For each object $N$ of $\D(A)$, define $$\Ext_A^*(M,N):=\h(\RHom_A(M,N))\text{ and }\Tor^A_*(M,N)=\h(M\ot_A N). $$
\end{chunk}

\begin{chunk}\label{augiso}
Let $\vp: A'\to A$ be a morphism of DG $Q$-algebras and let $M$ and $N$ be DG $A$-modules. By \cite[6.10]{FHT}, If $\vp$ is a quasi-isomorphism, then $\vp$ induces an isomorphism of graded $Q$-modules  $$\Ext_\vp^*(M,N): \Ext_A^*(M,N)\to \Ext_{A'}^*(M,N).$$
\end{chunk}

\begin{chunk}
Let $\T$ denote a triangulated category.  A full subcategory $\T'$ of $\T$ is called  triangulated if it is closed under suspension and  has the two out of three property on exact triangles. If, in addition, $\T'$ is closed under direct summands, we say that $\T'$ is a thick subcategory of $\T$. 

Let  $X$ be in $\T$. Define the thick closure of $X$ in $\T$, denoted $\Thick_\T X$, to be the intersection of all thick subcategories of $\T$ containing $X$. As an intersection of thick subcategories is a thick subcategory, $\Thick_{\T}X$ is the smallest thick subcategory of $\T$ containing $X$. See \cite[Section 2]{HPC} for an inductive definition of $\Thick_{\T}X $.
\end{chunk}

For the rest of the section, assume that $Q$ is a commutative noetherian ring and  $A$ is a non-negatively graded,  commutative DG $Q$-algebra such that $H_i(A)$ is finitely generated over $Q$ for each $i$, and the canonical map $Q\to H_0(A)$ is surjective. 

\begin{chunk}\label{LJ}  Suppose $(Q,\n,k)$ is  local and let $M$ be in $\D^f(A)$. In this context, we can test whether $M$ is in $\Thick_{\D(A)}A$ by calculating the eventual vanishing of certain graded Ext-modules. That is,    J{\o}rgensen established  in \cite[2.2]{PJor} that $M$ is in  $\Thick_{\D(A)} A$ if and only if $\Ext_{A}^{\gg 0}(M,k)=0$. \end{chunk}

\begin{chunk}\label{TJ}
 Let $M$ be a complex of $Q$-modules. The  amplitude of $M$ is defined to be 
$$\amp M:=\sup \{i: H_i(M)\neq 0\}-\inf\{i: H_i(M)\neq 0\}.$$  By  \cite[4.1]{PJor}, if $\amp A<\infty$, then for each nontrivial object $M$ of $\Thick_{\D(A)}A$ $$\amp A\leq \amp M.$$
\end{chunk}


\section{DG Modules Over a Graded Commutative Noetherian Ring}\label{s:coh}

In this section, we  follow the  convention in \ref{c:grade}  of grading objects cohomologically. The  reason being that theory in this section will be applied to studying DG modules and  graded Ext-modules over a ring of cohomology operators (see Section \ref{refco} for  details).

\subsection{Support} 

Let $\A=\{\A^{i}\}_{i\geq 0}$ be a  graded, commutative noetherian ring.  
Recall that, as a set, $\Proj \A$ consists of the homogeneous prime ideals of $\A$ not containing the irrelevant ideal $\A^{>0}:=\{\A^i\}_{i>0}.$  The topology on $\Proj \A$ is the Zariski topology, which has as its  closed sets those of the form $$\{\mathsf{p} \in \Proj \A: \mathsf{p}\supseteq \mathcal{I}\}$$ where $\mathcal{I}$ is a homogeneous ideal of $\A$.

\begin{chunk} We regard $\A$ as a DG algebra with trivial differential.  For a DG $\A$-module $\mathsf{X}=\{\mathsf{X}^i\}_{i\in \Z}$, $\h(X):=\{\h^i({\mathsf{X}})\}_{i\in \Z}$ is a graded $\A$-module and,  as usual,  $\mathsf{X}$ is an object of $\D^f(\A)$ if and only if $ \h(X)$ is a finitely generated graded $\A$-module. 
\end{chunk}

\begin{chunk}\label{localization}Let $\mathsf{X}$ be a DG $\A$-module. For each $\mathsf{p}\in \Proj \A$ we let $\mathsf{X}_{\mathsf{p}}$ denote the homogeneous localization of $\mathsf{X}$ at $\mathsf{p}$. We let $\kappa(\mathsf{p}):=\A_{\mathsf{p}}/{\mathsf{p}}\A_{\mathsf{p}}.$ By \cite[1.5.7]{BH}, $$\kappa(\mathsf{p})\cong k[t,t^{-1}]$$ for some field $k$ and a variable $t$ of   positive (cohomological) degree.
\end{chunk}

 \begin{definition}\label{defsup} The \emph{cohomological support} of a DG  $\A$-module $\mathsf{X}$ is defined to be  $$\gsupp{\A} \mathsf{X}:=\{\mathsf{p}\in \Proj \A:\mathsf{X}\ot_\A\kappa(\mathsf{p})\not\simeq 0\}.$$\end{definition}

\begin{chunk}\label{supform}
Let $S=k[t,t^{-1}]$  where $k$ is  a field and $t$ is a variable of positive cohomological degree. As $S$ is a graded field, for each DG $S$-module $X$ there exists a surjective homotopy equivalence $X\to \h(X)$ and we have an isomorphism of DG $S$-modules  $$\h(X)\cong  S^{(\beta)}$$  where $\beta$ is  a $k$-basis for $\h^0(X)$. Thus, there is a natural isomorphism $$\h(\Hom_S(X,-))\cong \Hom_S(\h(X),\h(-)).$$ Therefore, each DG $S$-module is semiprojective. 
\end{chunk}

 \begin{chunk}\label{tensor}
For DG $\A$-modules  $\mathsf{X}$ and $\mathsf{X}'$, $$\gsupp{\A}( \mathsf{X}\ot_\A  \mathsf{X}')=\gsupp{\A} \mathsf{X}\cap \gsupp{\A} \mathsf{X}'.$$ Indeed, for any homogeneous prime $\mathsf{p}$ we have the following  \begin{align*} \mathsf{X}\ot_\A  \mathsf{X}'\ot_\A\kappa(\mathsf{p})&\simeq ( \mathsf{X}\ot_\A\kappa(\mathsf{p}))\ot_{\kappa(\mathsf{p})}  (\mathsf{X}'\ot_\A\kappa(\mathsf{p})) \\
&\simeq( \mathsf{X}\ot_\A\kappa(\mathsf{p}))\otimes_{\kappa(\mathsf{p})}  (\mathsf{X}'\ot_\A\kappa(\mathsf{p}))\\ 
&\simeq \kappa(\mathsf{p})^{(\beta)}\otimes_{\kappa(\mathsf{p})} \kappa(\mathsf{p})^{(\beta')}. \end{align*} where $\beta$ and $\beta'$ are the $k$-bases for $\h^0(\mathsf{X})$ and $\h^0(\mathsf{X}')$, respectively (see  \ref{supform} for the last two isomorphisms).
\end{chunk}

\begin{chunk}\label{fgsup}
Let $\mathsf{X}$ be an object of $\D^f(\A)$. By \cite[2.4]{CI}, $$\gsupp{\A}\mathsf{X}=\{\mathsf{p}\in \Proj \A: M_{\mathsf{p}}\not\simeq 0\}.$$ Since localization is exact, $M_{\mathsf{p}}\simeq 0$ if and only if $\h(M)_{\mathsf{p}}=0$. Therefore, there is an equality \begin{align*}\gsupp{\A}\mathsf{X}&=\gsupp{\A}\h(X) \\
&=\{\mathsf{p}\in \Proj \A: \mathsf{p}\supseteq \text{ann}_\A(\h(\mathsf{X}))\}\end{align*}  where the second equality holds  by \cite[2.2(4)]{AI} . Thus, $\gsupp{\A}\mathsf{X}$ is a closed subset of $\Proj \A$ whenever $\mathsf{X}$ is an object of $\D^f(\A)$. 
\end{chunk}

 The next result follows easily from the definition of cohomological support (also see \cite[2.2]{AI}). 
\begin{proposition}\label{gradedsup}
Let  $\A=\{\A^{i}\}_{i\geq 0}$ be a cohomologically graded, commutative noetherian ring. 
\begin{enumerate}
\item Let $\mathsf{X}$ be a DG $\A$-module and $n\in \Z$. Then $\gsupp{\A}\mathsf{X}=\gsupp{\A}(\shift^n\mathsf{X}).$
\item  Let  $0\to\mathsf{X}'\to \mathsf{X}\to \mathsf{X}''\to 0$ be an exact sequence of DG $\A$-modules where either:
\begin{enumerate}
\item it is split exact, or
\item  each DG $\A$-module has \emph{trivial differential} and is an object of $\D^f(A)$.
\end{enumerate}
Then   $$\gsupp{\A} \mathsf{X}=\gsupp{\A} \mathsf{X}'\cup\gsupp{\A} \mathsf{X}''.$$
\item If $\mathsf{X}$ is an object of $\D^f(\A)$, then $\gsupp{\A}\mathsf{X}=\emptyset$ if and only if $\h^{\gg 0}(\mathsf{X})=0.$
\item Let $\mathsf{X}$ and $\mathsf{X'}$ be objects of $\D^f(\A)$ with trivial differential. Then $$\gsupp{\A}(\mathsf{X}\otimes_\A\mathsf{X'})=\gsupp{\A}\mathsf{X}\cap \gsupp{\A}  \mathsf{X'}.$$
\end{enumerate}
\end{proposition}


	\subsection{Finite Generation via Koszul Objects}
	\begin{chunk}\label{coseq}
Let $a\in \A$ be a homogeneous element of degree $d$. For a DG $\A$-module $\mathsf{X}$ we define  $$\kos{\mathsf{X}} a=\cone(\shift^{-d}\mathsf{X}\xra{a\cdot} \mathsf{X}).$$ As  $a\h(\kos{\mathsf{X}} a)=0$,  $\h(\kos{\mathsf{X}} a)$ is a graded $\A/a$-module. 
Also,  the following is an exact sequence of graded $\A$-modules  $$\shift^{-d}\h(\mathsf{X})\xra{a} \h(\mathsf{X})\to \h(\kos{\mathsf{X}} a)\to \shift^{-d+1}\h(\mathsf{X})\xra{a} \shift \h(\mathsf{X}).$$  \end{chunk}

\begin{chunk}\label{reg}
Let $\bm{a}:=a_1,\ldots, a_n\in \A$ be a sequence of homogeneous elements. For a DG $\A$-module $\mathsf{X}$, we define the \emph{Koszul object of $\mathsf{X}$ with respect to $\bm{a}$ } to be $$\kos{\mathsf{X}} {\bm{a}}:=((\kos{(\kos{\mathsf{X}}a_1)} a_2 )\ldots \sslash a_n).$$ It follows that $\h(\kos{\mathsf{X}} {\bm{a}})$ is a  graded $\A/(\bm{a})$-module. Furthermore, if $\bm{a}$ is a regular sequence on $\mathsf{X}^\natural$ then  we have the following isomorphism in $\D(\A)$ $$\kos{\mathsf{X}}{ \bm{a}} \xra{\simeq } \mathsf{X}/(\bm{a})\mathsf{X}.$$ 
\end{chunk}

\begin{chunk}\label{gradedNAK}
We recall  the \emph{graded version} of Nakayama's lemma: Let $\mathsf{M}$ be a graded $\A$-module such that $\mathsf{M}$ is bounded below,  i.e., $\mathsf{M}^{i}=0$ for $i \ll 0$. If $\A^{>0}\mathsf{M}=0$, then $\mathsf{M}=0$.  In particular, if $\mathsf{M}/\A^{>0}\mathsf{M}$ is a finitely generated graded $\A$-module, then $\mathsf{M}$ is a finitely generated graded $\A$-module. \end{chunk} 

\begin{theorem}\label{dgfg}
Let $\mathsf{X}$ be a DG $\A$-module such that $\h^{\ll 0}(\mathsf{X})=0$. For each sequence of homogeneous elements $\bm{a}:=a_1,\ldots, a_n\in \A$ of positive degree, $\h(\mathsf{X})$ is a finitely generated graded $\A$-module if and only if $\h(\kos{\mathsf{X}}{ \bm{a}})$ is a finitely generated graded $\A/(\bm{a})$-module. 
\end{theorem}
\begin{proof}
Since $\kos{\mathsf{X}} {\bm{a}}$ is defined inductively, it suffices to show that theorem holds when $\bm{a}=a$ is just a single homogeneous element $a$ of degree $d>0$. By \ref{coseq}, we have an exact sequence of DG $\A$-modules 
\begin{equation}\shift^{-d}\h(\mathsf{X})\xra{a} \h(\mathsf{X})\to \h(\kos{\mathsf{X}} a)\to \shift^{-d+1}\h(\mathsf{X})\xra{a} \shift \h(\mathsf{X}).\label{special}\end{equation}

Assume that $\h(X)$ is a finitely generated graded $\A$-module. Hence, (\ref{special}) implies that $\h(\kos{\mathsf{X}} a)$ is finitely generated over $\A$.  Since $a\h(\kos{\mathsf{X}} a)=0$, it follows that $\h(\kos{\mathsf{X}}{ \bm{a}})$ is a finitely generated graded $\A/({a})$-module.

Conversely,  assume that $\h(\kos{\mathsf{X}} a)$ is a finitely generated graded $\A/(a)$-module.  Since $a\h(\kos{\mathsf{X}} a)=0$ then $\h(\kos{\mathsf{X}} a)$ is finitely generated as a graded $\A$ via the canonical surjection $\A\to \A/ (a)$. From (\ref{special}), $\h(\mathsf{X})/a\h(\mathsf{X})$ is a submodule of  $\h(\kos{\mathsf{X}} a)$ and  since $\h(\kos{\mathsf{X}} a)$ is a noetherian graded $\A$-module, $\h(\mathsf{X})/a\h(\mathsf{X})$ is  a noetherian graded $\A$-module.  Therefore, $\h(\mathsf{X})/a\h(\mathsf{X})$ is  a  finitely generated graded $\A$-module. As $\h^{\ll 0}(\mathsf{X})=0$, it follows that $\left(\h(\mathsf{X})/a\h(\mathsf{X})\right)^{\ll 0}=0.$ Therefore, by \ref{gradedNAK}, it follows that $\h(\mathsf{X})$ is a finitely generated graded $\A$-module.  
\end{proof}


\section{Cohomological Operators}\label{refco}



\subsection{Koszul Complexes}\label{Koszul}

Fix a commutative noetherian ring $Q$. Let $\f=f_1,\ldots, f_n$ be a list of elements in  $Q$. Define $\text{Kos}^Q(\f)$ to  be the DG $Q$-algebra with $\text{Kos}^Q(\f)^\natural$ the exterior algebra on a free $Q$-module with basis $\xi_1,\ldots,\xi_n$ of homological degree 1, and differential $\del\xi_i=f_i.$ We  write $$\text{Kos}^Q(\f)=Q\langle \xi_1,\ldots,\xi_n|\del\xi_i=f_i\rangle.$$ 

\begin{chunk} \label{mindci}
Let $E:=\Kos^Q(\f)$. We say that $E$ is a \emph{derived complete intersection} if $Q$ is a regular ring. We say that $E$ is a \emph{minimal derived complete intersection} provided that $(Q,\n,k)$ is a regular {local} ring where $\f\con \n^2$ minimally generates $(\f)$. 
\end{chunk}
\begin{remark}\label{remark}
Imposing the condition that the base ring $Q$ is  regular  has strong implications which will be utilized heavily in Section \ref{s2}. There are two perspectives  which motivate the terminology in \ref{mindci}. 
First, when $\f$ is a $Q$-regular sequence then $E$ is a cofibrant replacement of the complete intersection ring  $Q/(\f)$ in $\D(Q)$. However, regardless whether $\f$ is assumed to be a $Q$-regular sequence, $E$ is still   an object of  $\D(Q)$ that possesses  the same homological properties as a complete intersection quotient ring of $Q$. Second, by taking the \emph{derived}  pushout of the following diagram 
 \begin{center}
 \begin{tikzcd}
\Z[x_1,\ldots, x_n]  \ar{r} \ar{d}  & Q \\ 
   \Z &  \end{tikzcd}
 \end{center} we obtain $E$, where $x_i\mapsto f_i$ along the horizontal map and the vertical map is the canonical quotient. In the special case that $\f$ is a $Q$-regular sequence, we get the complete intersection $Q/(\f)$. Either point of view suggests  derived complete intersections should naturally generalize of the    complete intersections rings. 
	
\end{remark}

\begin{chunk}
Let $(Q,\n,k)$ be a local ring. Suppose that $\f$ and $\f'$ are minimal generating set for the same ideal $I$. We have an isomorphism of DG $Q$-algebras $$\Kos^Q(\f)\xra{\cong } \Kos^Q(\f').$$  \end{chunk}

\begin{chunk}\label{kosfunc} Let  $$ E=Q\langle \xi_1,\ldots, \xi_n|\del \xi_i=f_i\rangle \text{ and }E'=Q'\langle \xi_1'\ldots, \xi_m'|\del \xi_i'=f_i'\rangle$$ be derived complete intersections. Assume further that  $Q$ and $Q'$ are both \emph{complete} regular local rings with $\h_0(E)=\h_0(E')$. There exists a derived complete intersection $E''$  and surjective DG algebra quasi-isomorphisms $E''\xra{\simeq}E$ and $E''\xra{\simeq} E'.$  

Indeed, we form the following commutative diagram with surjective ring morphisms
\begin{center}
\begin{tikzcd}
P
\arrow
[drr, bend left
, "\pi'"]
\arrow
[ddr, bend right
, "\pi",swap]
\arrow
[dr] & & \\
& Q\times_{\h_0(E)} Q'
\arrow
[r]
\arrow
[d]
& Q'
\arrow
[d] \\
& Q
\arrow
[r]
& \h_0(E)
\end{tikzcd}
\end{center} where $P$ is a regular local ring presenting the complete local ring  $Q\times_{\h_0(E)} Q'$. As $\pi$ and $\pi'$ are surjective maps between regular local rings, $\Ker \pi$ and $\Ker \pi'$ are generated by (linear) regular sequences; let $\x$ and $\x'$ minimally generate  $\Ker \pi$ and $\Ker \pi'$, respectively. Finally, we let $\g$ and $\g'$ be list of elements in $P$ such that $\g Q=\f$ and $\g' Q'=\f'$. 

First, since $(\x, \g)=(\x',\g')$ and $P$ is a local ring, we have an isomorphism of DG $P$-algebras $$\Kos^P(\x,\g)\xra{\cong} \Kos^P(\x',\g').$$ Moreover, since $\x$ and $\x'$ are regular sequences in $P$, we have that $$\Kos^P(\x, \g)\xra{\simeq} E \text{ and } \Kos^P(\x',\g')\xra{\simeq} E',$$ respectively.   Therefore,  $E'':=\Kos^Q(\x,\g)$ is a derived complete intersection with  surjective quasi-isomorphisms $E''\xra{\simeq}E$ and $E''\xra{\simeq} E'.$  
\end{chunk}

\begin{chunk}\label{kosaction}
Let $F\xra{\simeq} Q/(\f)$ be a DG $Q$-algebra resolution of $R$. We write  $$F_1=Q\xi_1'\oplus \ldots \oplus Q\xi_n'$$ where $\del^F(\xi_i')=f_i$. We have the following commutative diagram of DG $Q$-algebras\begin{center}
 \begin{tikzcd}
 & F\arrow["\simeq"]{d} \\ 
 E\arrow["\e"]{r} \arrow["\vp"]{ru} & Q/(\f) 
\end{tikzcd}\end{center} where $\e$ is the canonical augmentation map and $\vp$ is the morphism of DG $Q$-algebras determined by $$\xi_i\mapsto \xi_i'.$$ In particular, $F$ is a DG $E$-module via $$\xi_i\cdot x=\xi_i'x$$ for all $x\in F$. 
\end{chunk}

\begin{chunk}\label{selfdualkoszul}
We let $E:=\text{Kos}^Q(\f)$ for some list of elemenets $\f=f_1,\ldots,f_n$ in $Q$. There is a canonical DG $E$-module structure on $\Hom_Q(E,Q)$. Moreover, we have the following isomorphism of DG $E$-modules $$ \Hom_Q(E,Q)\cong \shift^{-n}E.$$ 
\end{chunk}

\begin{chunk}\label{kosgor}
Let $E=\Kos^Q(\f)$ for some list of elements $\f_1,\ldots, f_n$ in $Q$. Suppose $Q$ has finite injective dimension over itself and  set $$(-)^\dagger:=\RHom_E(-,\Hom_Q(E,Q)).$$ By \cite[2.1]{FIJ}, for each object $M$ of $\D^f(E)$ the following hold:
\begin{enumerate}
\item $M^\dagger$ is an object of $\D^f(E)$, and
\item the natural map $M\to M^{\dagger\dagger}$ is an isomorphism.
\end{enumerate}
By (2), it follows for each pair of objects $M$ and $N$ of $\D^f(E)$, \begin{equation}\RHom_E(M,N)\simeq \RHom_E(N^\dagger, M^\dagger).\label{goriso}\end{equation}
\end{chunk}

\begin{lemma}\label{derdual}
We let $E:=\Kos^Q(\f)$ for some list of elemenets $\f=f_1,\ldots,f_n$ in $Q$. For any object  $M$ of $\D(E)$,     $$\RHom_E(M,E)\simeq \shift^n \RHom_Q(M,Q).$$
\end{lemma}
\begin{proof}
By \ref{selfdualkoszul} and adjunction, we get the first two following isomorphisms  below
\begin{align*}\RHom_E(M,E)&\simeq\RHom_E(M,\shift^n \Hom_Q(E,Q)) \\
&\simeq \RHom_Q(M\ot_EE,\shift^n Q)\\
&\simeq \RHom_Q(M,\shift^n Q) \\
&\simeq \shift^n \RHom_Q(M,Q).
\end{align*} The last  two isomorphisms are standard. 
\end{proof}

\begin{proposition}\label{p:duality}
With the notation and assumptions in \ref{kosgor}, if  $M$ is an object of $\D(E)$ where $\h_i(M)$ is a finitely generated $Q$-modules for each $i$ and $\h_i(M)=0$ for all $i\ll 0$, then    $$M^{\dagger\dagger}\simeq M.$$
\end{proposition}
\begin{proof}First, we state a general isomorphism in (\ref{e:dd}), below. 
Let $X$ be a DG $E$-module. By \ref{selfdualkoszul} and  Lemma \ref{derdual}  the first and third isomorphisms below hold, respectively,   
\begin{align*}
X^\dagger&= \RHom_E(X,\Hom_Q(E,Q)) \\
&\simeq\RHom_E(X,\shift^{-n}E) \\
&\simeq \shift^{-n}\RHom_E(X,E) \\
&\simeq \shift^{-n}\shift^{n}\RHom_Q(X,Q) \\
&\simeq \RHom_Q(X,Q).
\end{align*} Hence, \begin{equation}\label{e:dd}X^{\dagger\dagger}\simeq \RHom_Q(\RHom_Q(X,Q),Q).\end{equation}

Returning the setup of the proof, by  assumption, there exists a semifree  resolution $F\xra{\simeq} M$ such that $$F^\natural=\coprod_{j=i}^\infty \shift^j (E^{\beta_j})^\natural.$$ 
Thus,  applying (\ref{e:dd}) yields 
\[M^{\dagger\dagger}\simeq F^{\dagger\dagger}
\simeq \RHom_Q(\RHom_Q(F,Q),Q ).\]
 Since $Q$ has finite injective dimension over itself and $F$ is bounded below complex of $Q$-modules that is  finitely generated (over $Q$) in each degree, it follows that $$\RHom_Q(\RHom_Q(F,Q),Q ) \simeq F$$ (see \cite[A.4.24]{Lars}). Hence,   $$M^{\dagger\dagger}\simeq  F\simeq M.\qedhere$$
\end{proof}

	\begin{Construction}\label{const1}
	 Suppose $E=\text{Kos}^Q(\f)$ 
	 and   set $\Poly:=Q[\chi_1,\ldots, \chi_n]$ to be  the polynomial ring  in $n$  variables of cohomological degree two. As a graded $Q$-algebra, $\Poly$ has a graded basis   \begin{equation}\label{e:basis}\{\chi^H:=\chi_1^{h_1}\ldots \chi_n^{h_n}\}_{H\in \N^n}.\end{equation}
	 We let  $\Gamma$  be the graded $Q$-linear dual of $\Poly$, i.e., $$	\Gamma:=\Hom_Q^*(\Poly,Q)=\bigoplus_{i\in \N} \Hom_Q(\Poly^i,Q).$$ Let  $\{y^{(H)}\}_{H\in \N^n}$ be the graded $Q$-basis  of ${\Gamma}$  which is dual to (\ref{e:basis}). We will regard $\Gamma$ as a  graded $\Poly$-module with structure  determined  by
	 \[
	 \chi_i \cdot y^{(H)}=\left\{\begin{array}{cl}  y^{(h_1,\ldots,h_{i-1}, h_i-1,h_{i+1}, \ldots,h_n)} & h_i\geq 1\\
	 0 & h_i=0.
	 \end{array}\right.
	 \]

	 We let  $E^e:=E\otimes_QE^o$ where $E^o$ denotes the opposite DG algebra of $E$ and 
	 we regard $E$ as a DG $E^e$-module via the multiplication map $\mu: E^e\to E.$   
	By \cite[2.6]{CD2},  $E$ admits a semiprojective  resolution $L$ over $E^e$ defined in the following way
 	 \begin{align*} &L^\natural :=(E\otimes_Q \Gamma\otimes_Q E)^{\natural}\\  
	 \del^{L}=\del^{E}\otimes 1\otimes 		1+&1\otimes1\otimes \del^{{E}}+\sum_{i=1}^n (1\otimes \chi_i\otimes \lambda_i-\lambda_i\otimes \chi_i\otimes 	1)\end{align*} where $\lambda_i: E\to E$ is left multiplication by $\xi_i,$  and  augmentation map $\e: L\to E$  is given by $$a\otimes y^{(H)}\otimes b\mapsto \left\{\begin{array}{cl}  ab 	& |H|=0 \\ 0 & |H|>0 \end{array}\right..$$
	\end{Construction}


	\subsection{Universal Resolutions}\label{UKres}
	
Throughout this section,  $\f=f_1,\ldots, f_n$ is a list of elements in a commutative noetherian ring $Q$, $E=Q\langle \xi_1,\ldots, \xi_n|\del \xi_i=f_i\rangle,$ and $L$ is the semiprojective resolution of $E$ over $E^e$ defined  in Construction \ref{const1}. 

\begin{chunk}\label{Koszulres}
Let $M$ be a DG $E$-module.  By  \ref{sres} (or \cite[2.1]{CD2} in this specific context), there exists a quasi-isomorphism of DG $E$-modules  $F\xra{\simeq} M$ where   $F$ is semiprojective as a DG $Q$-module when regarded as a DG $Q$-module via restriction of scalars along $Q\to E$. We call such a  resolution a \emph{Koszul resolution of M}. Moreover, if $M$ is an object of $\Thick_{\D(Q)}Q$ when regarded as a complex of $Q$-modules, then there exists a Koszul resolution $F$ where each $F_i$ is  finitely generated as a $Q$-module and $F_i=0$ for all $|i|\gg 0.$  Such a Koszul resolution is called \emph{finite}. 
\end{chunk}
	
	\begin{chunk}\label{univres}
	 Let $M$ be a DG $E$-module. Fix a Koszul resolution $\delta:F\xra{\simeq} M$ and define  $$U_E(F):=L\otimes_E F,$$ with   augmentation map $\e^M: U_E(F)\to M$  given by 
  $$a\otimes y^{(H)}\otimes x\mapsto \left\{\begin{array}{cl}  a\delta(x) & |H|=0 \\ 0 & |H|>1 \end{array}\right.$$
	  By \cite[2.4]{CD2},  $\e^M: U_E(F)\xra{\simeq} M$  is a semiprojective DG $E$-module resolution of $M$ . We call $U_E(F)$ the \emph{universal resolution of M over E (with respect to $\delta:F\to M$)} which is a  DG $\Poly$-module via the  $\Poly$-action on $\Gamma$. 
\end{chunk}
	 
	 \begin{chunk} \label{c:polynomialaction}
 Let $M$ and  $N$ be DG $E$-modules and let $U$ be a universal resolution of M over E  associated to the Koszul resolution $F\xra{\simeq }M$ (see \ref{univres}). Since  $U$ is a  DG $\Poly$-module, it follows that   $\Hom_{E}(U,N)$ is  a DG $\Poly$-module. Hence, $$\Ext_{E}^*(M,N)=H^*(\Hom_{E}(U,N))$$ is a graded module over $\Poly$. The $\Poly$-module structure on $\Ext_{E}^*(M,N)$ is independent of choice of Koszul resolution in the construction above (see  \cite[3.2.2]{Pol}).\end{chunk}

\begin{proposition}\label{graded}
Let $M$ and $N$ be DG $E$-modules.  The $\Poly$-module structure on $\Ext_E^*(M,N)$ is independent of choice of Koszul resolution for $M$. Moreover, the $\Poly$-module action on $\Ext_E^*(M,N)$ is functorial in both $M$ and $N$. Namely, we have the following isomorphisms of graded $\Poly$-modules. 
\begin{align*}\Ext_E^*(M,\shift N)\cong \shift&\Ext_E^*(M,N)\cong \Ext_E^*(\shift^{-1}M,N), \\
\Ext_E^*(M,N\oplus N')&\cong \Ext_E^*(M,N)\oplus \Ext_E^*(M,N'),\text{ and }  \\
\Ext_E^*(M\oplus M', N)&\cong \Ext_E^*(M,N)\oplus \Ext_E^*(M',N)\end{align*}
and for each exact triangle $M'\to M\to M''\to $ in $\D(E)$ and each object $N$ of $\D(E)$ we get an exact sequences of graded $\Poly$-modules: 
\begin{align*}\Ext_E^*(M'',N)&\to \Ext_E^*(M,N)\to \Ext_E^*(M',N) \to \shift \Ext_E^*(M'',N) \text{ and } \\
\Ext_E^*(N,M')&\to \Ext_E^*(N,M)\to \Ext_E^*(N,M'') \to \shift \Ext_E^*(N,M').\end{align*}
\end{proposition}
\begin{proof}
The first part of the proposition and the functoriality of the first entry of $\Ext_E^*(-,-)$ was shown in \cite[3.2.2]{Pol}. The functoriality in the second component of $\Ext_E^*(-,-)$ is an immediate consequence of  the $\Poly$-module structure being determined by the first entry. \end{proof}

\begin{chunk} \label{haction}
By \cite[2.9]{CD2}, we have an isomorphism of graded $Q$-algebras \begin{equation}\vp: \hh^*(E|Q):=\h(\Hom_{E^e}(L,E))\xra{\cong} \h(E)[\chi_1,\ldots,\chi_n]\label{eiso}\end{equation} where each $\chi_i$ is a cohomologically graded variable of degree two.   
Using the isomorphism of complexes $$\psi^{F,N}:\Hom_{E^e}(L, \Hom_Q(F,N))\xra{\cong}\Hom_{E}(U,N)$$ where $F$ is a semiprojective  resolution of $M$ over $Q$, it follows this $\Poly$-module action is compatible with the action of $\hh^*(E|Q)$ on $\Ext_{E}^*(M,N).$ That is,  $$\h(\psi^{F,N})( \alpha\cdot h)=\h(\psi^{F,N})(\alpha) \cdot\vp(h)$$ for every homotopy class $\alpha$ of a cycle from $\Hom_{E^e}(L, \Hom_Q(F,N))$ and $h\in \hh^*(E|Q).$
\end{chunk}

\begin{remark}
A specific  version of the following construction first appeared in \cite[3.2]{CD2} when $\f$ is a Koszul-regular sequence and $M$ and $N$ are $Q/(\f)$-modules. The differential below and the coefficients of $\Poly$ have been adjusted from its form  in \cite[3.2]{CD2} to accommodate for $N$ being a DG $E$-module instead of a $Q/(\f)$-module. 
\end{remark}

\begin{Construction}\label{c:dgsmod}
Let $X$ and $Y$ be DG $E$-modules. Define $\CC_E(X,Y)$ to be   \begin{align*}\CC_E(X,Y)^\natural& :=\Poly\otimes_Q\Hom_Q(X,Y)^\natural\\  \del^{\CC_E(X,Y)}:=1\otimes \del^{\Hom_Q(X,Y)}&+\sum_{i=1}^n \chi_i\otimes (\Hom(\lambda_i,Y)-\Hom(X,\lambda_i)).
\end{align*} A direct calculation shows that $\CC_E(X,Y)$ is a DG $\Poly$-module. 
\end{Construction}
\begin{proposition}\label{altform}
Let $Q$ be a commutative noetherian ring, $\f=f_1,\ldots, f_n\in Q$ and $E$ be the Koszul complex on $\f$. Suppose $M$ and $N$ are DG $E$-modules and let $F\xra{\simeq} M$ be a Koszul resolution of $M$. Then we have an isomorphism of DG $\Poly$-modules $$\Hom_{E}(U_E(F),N)\cong \CC_E(F,N).$$ 
\end{proposition}
\begin{proof}
Consider the following isomorphisms of graded $\Poly$-modules
\begin{align*}
\Hom_{E}(U_E(F),N)^\natural  &\cong \Hom_E(E\otimes_Q \Gamma, \Hom_Q(F,N))^\natural\\
&\cong \Hom_{E}(E\otimes_Q \Gamma, E)\otimes_{E} \Hom_Q(F,N)^\natural \\
&\cong  \Hom_{Q}(\Gamma, Q)\otimes E\otimes_{E} \Hom_Q(F,N)^\natural \\
&\cong \Poly \otimes_Q\Hom_Q(F,N)^\natural.
\end{align*} Tracing the differential of $\Hom_{E}(U_E(F),N)$  through these isomorphisms completes the proof. 
\end{proof}

\begin{remark}\label{complex}
Let $X$ be DG $E$-module and assume that $N$ is a $Q/(\f)$-module  such that $\del^{\Hom(X,N)}=0$.  In this case,  $\CC_E(X,N)$ reduces to  the following complex of graded $\Poly$-modules,  $$\CC_E^{j}(X,N)=\shift^{2j}\Poly\otimes_Q \Hom_Q(X_{-j}, N)$$ and the differential is   $$\sum_{i=1}^n\chi_i\otimes \Hom(\lambda_i,N).$$ When $X\xra{\simeq}M$ is a Koszul resolution then  by Proposition \ref{altform}, the $\Poly$-module structure of $\Ext_E^*(M,N)$ can be calculated as the cohomology of the $\CC_E(X,N)$. This was first noticed   in \cite[3.7]{CD2}.
\end{remark}

\begin{proposition}\label{underivetensor}
Let $E$ be a derived complete intersection and let $M$ and $N$ be objects of $\D^f(E)$. For any   pair of Koszul resolutions $F\xra{\simeq}M$ and $G\xra{\simeq} N$,  $\CC_E(F,G)$ is a semiprojective DG $\Poly$-module. 
\end{proposition}
\begin{proof}
By assumption $Q$ is regular, and hence, $\Poly$ is  regular. Also, $\CC_E(F,G)^\natural$ is a free graded $\Poly$-module. Thus, by Proposition \ref{gpsp}, we conclude that $\CC_E(F,G)$ is a semiprojective DG $\Poly$-module, as claimed. 
\end{proof}

\begin{remark}
Let $E:=\Kos^Q(f_1,\ldots, f_n)$, where $Q$ is not necessarily a regular  ring, and set $\A:=\Poly\otimes_Q k$. It follows that $\CC_E(F,G)\otimes_Q k$ is a semiprojective DG $\A$-module for  DG $E$-modules $F$ and $G$ such that $F$ and $G$ are graded projective when regarded as DG $Q$-modules. Indeed, since $F$ and $G$ are graded projective, it follows that $\Hom_Q(F,G)$ is a complex of projective $Q$-modules. Thus, $\CC_E(F,G)\otimes_Q k$ is a graded projective DG $\A$-module. Finally, since $\A$ has finite global dimension, the claim follows from Proposition \ref{gpsp}. 
\end{remark}



\subsection{Finite Generation}

In this section, we establish analogs to the theorems of Gulliksen and Avramov-Gasharov-Peeva (see \cite[3.1]{G} and \cite[4.2]{AGP}, respectively). Theorem \ref{Gul} specializes to  the aforementioned results. In the following proposition, 
we give an elementary argument for one implication of Theorem \ref{Gul} that holds provided that $E$ is a derived complete intersection; this is essentially the same  argument from \cite[4.5]{AG} transported to the derived setting.

\begin{proposition}\label{p:fg}
Let $E$ be a derived complete intersection. For each pair of objects $M$ and $N$ of $\D^f(E)$, $\Ext_E^*(M,N)$ is a finitely generated graded $\Poly$-module. 
\end{proposition}
\begin{proof}
By \cite[2.1]{CD2}, since $Q$ is regular there exist  finite Koszul resolutions $P\xra{\simeq} M$ and  $P'\xra{\simeq} N$.
Thus, $\CC_E(P,P')^\natural$ is a noetherian graded $\Poly$-module. Hence, any subquotient of $\CC_E(P,P')$ is a noetherian graded $\Poly$-module. In particular, $$\h(\CC_E(P,P'))\cong \Ext_E^*(M,N)$$ is a noetherian graded $\Poly$-module.
\end{proof}

\begin{theorem}\label{GC}\label{Gul}Let $Q$ be a commutative noetherian ring and $\f=f_1,\ldots, f_{n}$ be a sequence of elements  in $Q$. Define  $E$ to be the Koszul complex on $\f$ and  let $\Poly=Q[\chi_1,\ldots, \chi_{n}]$ where each  $\chi_i$ has  cohomological degree two. 
 For a pair of objects $M$ and $N$  of $ \D^f(E)$,  $\Ext_Q^*(M,N)$ is a finitely generated graded $Q$-module if and only if $\Ext_E^*(M,N)$ is a finitely generated graded $\Poly$-module. 
\end{theorem}
\begin{proof}
Let $F\xra{\simeq} M$ and $G\xra{\simeq} N$ be  Koszul resolutions of $M$ and $N$, respectively.  By Theorem \ref{dgfg}, $\Ext_E^*(M,N)$ is a finitely generated   graded $\Poly$-module if and only if $\kos{\h(\CC_E(F,G)}{ \bm{\chi}})$ is a finitely generated graded $\Poly/(\bm{\chi})=Q$-module where $\bm{\chi}=\chi_1,\ldots, \chi_n.$  

Next, since $F$ and $G$ are free as graded $Q$-modules it follows that $\CC_E(F,G)^\natural$ is a free graded $\Poly$-module. Thus, $\bm{\chi}$ is a regular sequence on $\CC_E(F,G)$. Therefore, we have an isomorphism in $\D(S)$,  $$\kos{\CC_E(F,G)}{\bm{\chi}}\xra{\simeq} \CC_E(F,G)/(\bm{\chi}) \CC_E(F,G)$$ (see \ref{reg}). Finally, we need only observe that $$\CC_E(F,G)/(\bm{\chi}) \CC_E(F,G)\cong \Hom_Q(F,G).\qedhere$$
\end{proof}



\subsection{Functoriality}

\begin{chunk}Throughout this section, $Q$ and $Q'$  will be commutative noetherian rings. Fix $\f=f_1,\ldots, f_n\in Q$ and $\f':=f_1',\ldots, f_m'\in Q'.$ Define $E$ and $E'$ to be the Koszul complexes on $\f$ and $\f'$, respectively. That is, $$E:=Q\langle \xi_1,\ldots, \xi_n|\del \xi_i=f_i\rangle \text{ and }E':=Q'\langle \xi_1',\ldots, \xi_m'|\del \xi_i'=f_i'\rangle.$$  
\end{chunk}

In the following discussion, we use the theory of  $\Gamma$-DG algebras; see \cite[Section 6]{IFR} or \cite[Chapter 1]{GL} for the necessary background. 

	\begin{Discussion}\label{disc}Let $L$ and $L'$ be as in Construction \ref{const1}. Then $L$ is a semifree DG $\Gamma$-extension of $E^e$ presented by $$L\cong E^e\langle y_1,\ldots, y_n|\del y_i=1\otimes \xi_i-\xi_i\otimes 1\rangle.$$  Similarly,  $L'$ is a semifree DG $\Gamma$-extension of $(E')^e$ presented by $$L'\cong (E')^e\langle y_1',\ldots, y_m'|\del y_i'=1\otimes \xi_i'-\xi_i'\otimes 1\rangle.$$

	Suppose that $\vp: Q\to Q'$ is a morphism of rings such that $\vp(\f)\con \f'$. For each $1\leq i \leq n$ there exists $q_{ij}'\in Q'$ such that $$\vp(f_i)=\sum_{j=1}^m q_{ij}'  f_j'.$$ This induces a DG algebra map $\tilde{\vp}: E\to E'$ extending $\vp$ where $$\tilde{\vp}(\xi_i)=\sum_{j=1}^m q_{ij}'  \xi_j'.$$  Therefore, we have a DG algebra map $\psi: E^e\to L'$ given by the composition $$E^e\xra{\tilde{\vp}\otimes \tilde{\vp}} E'\otimes_Q E'\to (E')^e\inca L'.$$
	Observe that 
	 \begin{align*}\psi(\del(y_i))&=\psi(1\otimes \xi_i-\xi_i\otimes 1) \\
	&=1\otimes \sum_{j=1}^m q_{ij}'  \xi_j'-\sum_{j=1}^m q_{ij}'  \xi_j'\otimes 1 \\
	&=\sum_{j=1}^m q_{ij}' (1\otimes \xi_j'-\xi_j'\otimes 1) \\
	&=\sum_{j=1}^m q_{ij}' \del(y_j') \\
	&=\del\left(\sum_{j=1}^m q_{ij}' y_j'\right).
	\end{align*}
	Therefore, there exists a DG $\Gamma$-algebra map $\Psi: L\to L'$ extending $\psi$ such that $$\Psi(y_i)=\sum_{j=1}^m q_{ij}' y_j'.$$ Moreover, $\Psi$ is unique up to homotopy. 

The composition $\chi_i'\Psi: L\to L'$ is a  DG $\Gamma$-algebra map of homological degree -2 that is determined by its image on $\y$. Observe that $$\chi_i' \Psi (y_h)=\chi_i'\sum_{j=1}^m q_{hj}' y_j'=q_{hi}'.$$ Also, $$\sum_{j=1}^nq_{ji}' \Psi( \chi_j y_h)=q_{hi}'.$$ Thus, $$\chi_i' \Psi =\sum_{j=1}^n q_{ji}' \Psi\circ \chi_j .$$

Thus, for the natural transformation $$\Hom(\Psi,-):\Hom_{(E')^e}(L',-)\to \Hom_{E^e}(L,-)$$ we have that  \begin{equation}\label{comp1} \Hom(\Psi,-)\circ \chi_i'=\sum_{j=1}^n \chi_j\circ \Hom(\Psi, -)\circ q_{ji}' .\end{equation} If there exists $q_{ij}\in Q$ such that $\vp(q_{ij})=q_{ij}'$ for each $i,j$ then \begin{equation}\label{comp2} \Hom(\Psi,-)\circ \chi_i'=\sum_{j=1}^n q_{ij}\chi_j\circ \Hom(\Psi, -).\end{equation}
\end{Discussion}
\begin{proposition}\label{compare}
With the notation of Discussion \ref{disc},  let $M$ and $N$ be DG $E$-modules and $M'$ and $N'$ be DG $E'$-modules. Suppose $\alpha: M\to M'$ and $\beta: N'\to N$ are $E$-linear maps. Then $$\Ext_{\tilde{\vp}}^*(\alpha, \beta)\circ \chi_i'=\sum_{j=1}^n \chi_j\circ \Ext_{\tilde{\vp}}^*(\alpha, \beta)\circ q_{ji}'$$ where $\Ext_{\tilde{\vp}}^*(\alpha, \beta)$ is the canonical map $$\Ext_{E'}^*(M',N')\to \Ext_E^*(M,N).$$ Moreover, if there exists $q_{ij}\in Q$ such that $\vp(q_{ij})=q_{ij}'$ for each $i,j$ then  $$ \Ext_{\tilde{\vp}}^*(\alpha, \beta)\circ \chi_i'=\sum_{j=1}^n q_{ij}\chi_j\circ \Ext_{\tilde{\vp}}^*(\alpha, \beta).$$ 
\end{proposition}
\begin{proof}
Recall that $$\Ext_{E^e}^*(E,\RHom_Q(M,N))\cong \Ext_E^*(M,N).$$ This combined with  (\ref{comp1}) and (\ref{comp2}) yield the desired results. 
\end{proof}

\section{Cohomological Support}\label{cs}


\begin{Not}Throughout the entirety of  Section \ref{cs}, we fix the following notation. Let $Q$ be a commutative noetherian ring,  $\f=f_1,\ldots, f_n$  a list of elements in $Q$ and set $E:=\Kos^Q(\f)$.   We let $\Poly:=Q[\chi_1,\ldots, \chi_n]$  be the ring of cohomology operators where each $\chi_i$ has cohomological degree two.  Finally, let $\PS_Q^{n-1}$ denote   $\Proj \Poly$ equipped with the Zariski topology.
\end{Not}



\subsection{Cohomological Support  for   Pairs of  DG $E$-modules}\label{s1}

This section introduces  the theory of support for pairs of DG modules over the Koszul complex $E$. Recall that for each pair of DG $E$-modules $M$ and $N$,  $\Ext_E^*(M,N)$ naturally inherits the structure of a graded $\Poly$-module (see Section \ref{refco}).  That is, $\Ext_E^*(M,N)$ is a DG $\Poly$-module with trivial differential.  We will be importing the support theory from Section \ref{s:coh}, below.

\begin{definition}
Let $M$ and $N$ be DG $E$-modules.  We define the \emph{cohomological support of $M$ and $N$ over $E$} to be $$\V_E(M,N):=\gsupp{\Poly}(\Ext_E^*(M,N))\con \PS_Q^{n-1}.$$ The \emph{cohomological support  of $M$ over $E$} is defined to be $\V_E(M):=\V_E(M,M)$. 
\end{definition}

\begin{proposition}\label{bsp}
Let $M$ and $N$ be DG $E$-modules.
\begin{enumerate}
\item $\V_E(\shift^i M,N)=\V_E(M,N)=\V_E(M,\shift^i N)$ for each $i\in \Z$.
\item If $M=M'\oplus M''$ in $\D(E)$ then \begin{enumerate}
\item $\V_E(M,N)=\V_E(M',N)\cup\V_E(M'',N).$
\item $\V_E(N,M)=\V_E(N,M')\cup\V_E(N,M'').$
\end{enumerate}
\item For each exact triangle $M^1\to M^2\to M^3\to$ in $\D^f(E)$, we have \begin{enumerate}\item $\V_E(M^h,N)\con\V_E(M^i, N)\cup\V_E(M^j,N)$ for $\{h,i,j\}=\{1,2,3\}.$ 
\item $\V_E(N,M^h)\con\V_E(N,M^i)\cup\V_E(N,M^j)$ for $\{h,i,j\}=\{1,2,3\}.$
\end{enumerate}
\item If $M$ is an object of $\Thick_{\D(E)}M'$ then $$\V_E(M,N)\con \V_E(M',N)\text{ and }\V_E(N,M)\con \V_E(N,M').$$
\item $\V_E(M,N)\con \V_E(M,M)\cap \V_E(N,N)$.
\item If $\Ext_E^*(M,N)$ is bounded, then $\V_E(M,N)=\emptyset$. \end{enumerate}
\end{proposition}
\begin{proof} Parts (1), (2) and (3) are easy consequences of Proposition \ref{gradedsup} and  \ref{graded}. We leave these as exercises to the reader.

For (4), the full subcategory $\T$ of $\D(E)$ consisting of all objects $X$ such that $\V_E(X,N)\con \V_E(M',N)$ is a thick subcategory of $\D(E)$ since (1), (2) and (3) hold. Thus, since $M'$  is an object of $\T$ and $M$ is an object of $\Thick_{\D(E)}M'$ we have that $M\in \T$.  Therefore, this proves the first containment in (4), and the second containment  is similar. 

The $\Poly$-action on $\Ext_E^*(M,N)$ is central by \ref{haction} in the sense that it is compatible with the $\Ext_E^*(N,N)$-$\Ext_E^*(M,M)$ bimodule action. Now (5) is an immediate consequence. 

Finally, for (6), the first part is obvious as a high enough power of $\Poly^{>0}$ annihilates $\Ext_E^*(M,N)$.  \end{proof}

Unsurprisingly, when $\Ext_E^*(M,N)$ is  \emph{finitely generated} this support theory satisfies the following important properties that will be used heavily in the sequel. 
\begin{proposition}\label{refg}
Assume that $M$ and $N$ are objects of $\D^f(E)$ and $\Ext_Q^{\gg 0}(M,N)=0$. 
\begin{enumerate}
\item $\V_E(M,N)$ is a Zariski closed subset of $\PS_Q^{n-1}$.
\item $\V_E(M,N)=\gsupp{\Poly}(\CC_E(F,N))$ for any Koszul resolution $F\xra{\simeq }M$. 
\item $\V_E(M,N)=\emptyset$ if and only if $\Ext_E^{\gg 0}(M,N)=0$.
\end{enumerate}
\end{proposition}
\begin{proof}
The assumption that $\Ext_Q^{\gg 0}(M,N)=0$ is equivalent to $\Ext_E^*(M,N)$ being finitely generated over $\Poly$-module using Theorem \ref{Gul}. Hence, $\V_E(M,N)$ is a Zariski closed subset of $\PS_Q^{n-1}$, finishing the proof of (1).  Moreover,   for any Koszul resolution $F\xra{\simeq}M$, $\CC_E(F,N)$ is an object of $\D^f(\Poly)$ where 
\[\h(\CC_E(F,N))=\Ext_E^*(M,N).\]  Thus, by Proposition \ref{gradedsup}(3) 
\[\gsupp{\Poly}(\CC_E(F,N))=\gsupp{\Poly}(\Ext_E^*(M,N))=\V_E(M,N)\] and so (2) holds. 

Finally, for (3), we note the backwards implication is Proposition \ref{bsp}(6). For the converse, since  $\Ext_E^*(M,N)$ is a finitely generated graded $\Poly$-module then applying Proposition \ref{gradedsup}(3) yields the desired result. 
\end{proof}

\begin{example}\label{example1} \ 
\begin{enumerate}
\item Let $M$ be an object of $\Thick_{\D(E)}E$.  For each object $N$ of $\D^f(E)$,   $$\V_E(M,N)=\emptyset.$$ 
Indeed, $\Ext_E^*(E,N)=\h(N)$ and since $N$ is an object of $\D^f(E)$, it follows that $\gsupp{\Poly}(\Ext_E^*(E,N))=\emptyset$. That is, $\V_E(E,N)=\emptyset$ and so applying Proposition \ref{bsp}(4) yields the desired result. 

\item Let $E$ be a minimal derived complete intersection (see \ref{mindci}) with residue field $k$. By \cite[3.2.6]{Pol}, there is an isomorphism of graded $\Poly$-modules $$\Ext_E^*(k,k)\cong \Poly\otimes_Q \bigwedge (\shift^{-1} k^{\dim Q}).$$ In particular, if $\iota: \PS_k^{n-1}\inca \PS_Q^{n-1}$ is the canonical map  induced by the canonical quotient map of $\Poly\to k[\chi_1,\ldots, \chi_n]$, then $$\V_E(k,k)=\iota(\PS_k^{n-1}).$$ That is, $\V_E(k,k)=\gsupp{\Poly}(\Poly\otimes_Qk).$ 
\end{enumerate}
\end{example}

\begin{proposition}\label{dual} 
With the notation and assumptions from \ref{kosgor}, we let  $M$ be an object of  $\D^f(E)$.
\begin{enumerate}
\item $\V_E(M,N)=\emptyset$ for each object $N$ of $\Thick_{\D(E)}E$.
\item $\V_E(M,N)=\V_E(N^\dagger, M^\dagger)$ for each object $N$ of $\D^f(E)$.
\item $\V_E(M)=\V_E(M^\dagger)$.
\end{enumerate}
\end{proposition}
\begin{proof}
For (1), we note that $\Ext_E^*(M,E)$ is bounded by \ref{kosgor}(1). Hence, by Proposition \ref{bsp}(6), $\V_E(M,E)=\emptyset$. Now applying Proposition \ref{bsp}(4) yields $\V_E(M,N)=\emptyset$, as claimed. Next, (\ref{goriso}) implies that \[\Ext_E^*(M,N)\cong \Ext_E^*(N^\dagger,M^\dagger)\] from which (2) follows from. Finally, (3) is obtained from (2) by letting $N=M$. 
\end{proof}


\subsection{Local Base Ring}\label{s2}

Throughout this section we add the assumption that $Q$ is local with maximal ideal $\n$ and residue field $k$. 
One of the notable features of these cohomological supports is that when $Q$ is local they  detect whether a DG $E$-module is an object of $\Thick_{\D(E)}E$.  \begin{proposition}\label{prop2}
Assume  $M$ is an object of $\D^f(E)$ such that $\Ext_Q^{\gg 0}(M,k)=0$.  Then $\V_E(M,k)=\emptyset$ if and only if $M$ is an object of $\Thick_{\D(E)}E.$ 
\end{proposition}
\begin{proof}
 By  assumption and Proposition \ref{refg}(3) we conclude that  $\V_E(M,k)=\emptyset$ if and only if $\Ext_E^{\gg 0}(M,k)=0.$ Since  $M$ is an object of  $\D^f(E)$ with $\Ext_E^{\gg 0}(M,k)=0$ then by  \ref{LJ}, it follows that  $\Ext_E^{\gg 0}(M,k)=0$ if and only if $M$ is an object of  $\Thick_{\D(E)}E.$ 
\end{proof}

\begin{theorem}\label{cmsv}
 Set  $R:=Q/(\f)$ and  suppose that $\pd_Q R<\infty$. The following are equivalent:
\begin{enumerate}
\item $\f$ is a $Q$-regular sequence.
\item $\V_E(R,k)=\emptyset$
\item $\V_E(M,k)=\emptyset$ for some nonzero  finitely generated $R$-module $M$ such that  $\pd_Q M<\infty$. 
\end{enumerate}
\end{theorem}\begin{remark}
The equivalence of (1) and (2) was first  established in \cite[3.4(3)]{Pol}.  As noted in \cite[3.4(3)]{Pol}, ``(1) implies (2)" is clear from \ref{augiso} since the augmentation map $E\to R$ is a quasi-isomorphism. Moreover, there is nothing to show for  ``(2) implies (3)" and  so that leaves only the following implication to prove.  
\end{remark}
\begin{proof}[Proof of $(3)\implies (1)$]  Assume $\V_E(M,k)=\emptyset$ where $M$ is as in (3). By Proposition $\ref{refg}(3)$,  $\Ext_E^{\gg0}(M,k)=0$ and so by \ref{LJ},    $M$ is in  $\Thick_{\D(E)}E$. Thus,  \ref{TJ} yields \[\amp E\leq \amp M=0.\] Therefore, $E\xra{\simeq} {R}$ and hence, $\f$ is a $Q$-regular sequence.
\end{proof}

For the rest of the section we set $$\A:=\Poly\otimes_Qk=k[\chi_1,\ldots,\chi_n].$$

\begin{definition}\label{defsupportintro}
Let $M$ and $N$ be  objects of $\D(E)$. We define $$\VV_E(M,N):=\gsupp{\A} (\Ext_E^*(M,N)\otimes_Qk).$$ Set $$\VV_E(M):=\VV_E(M,M).$$
\end{definition}

\begin{chunk} Notice that $\VV_E(M,N)$ is the closed fiber of $\V_E(M,N)$, that is, $$\VV_E(M,N)=\V_E(M,N)\times_{\Spec Q}\Spec k\con \PS_k^{n-1}.$$
Thus, the formulas in Sections \ref{s1}, \ref{s2} and \ref{s3} all hold for $\VV_E(-,-)$ in place of $\V_E(-,-)$. To save space, we do not list the corresponding formulas for $\VV_E(-,-)$.\label{c:v2}
\end{chunk}

\begin{remark}\label{remarkdgi}
 In \cite{Pol}, the present author defined the supports $\VV_E(M,k)$ and provided an application for them. As briefly referenced in the introduction, this theory of supports is to used, in conjunction with Theorem \ref{cmsv}, to establish the following derived category characterization of a complete intersection ring: \emph{a commutative noetherian ring $R$ is locally a complete intersection if and only if every object of $\D^f(R)$ is virtually small}.  The main ingredients are  to show there exist objects $C(1),\ldots, C(n)$ in $\Thick_{\D(R)}k$ such that $$\VV_E(C(1))\cap\ldots  \cap\VV_E(C(n))=\emptyset$$ (cf. \cite[3.3.4]{Pol}), and   if $M$ is virtually small then $\VV_E(R)\subseteq \V_E(M)$. We also point out a similar application has recently appeared in a collaboration of the present author with Briggs and Grifo in \cite{BEJ}.
\end{remark}

\begin{lemma}\label{lt}
Let $E$ be a derived complete intersection and let  $M$ and $N$ be a pair of objects in $\D^f(E)$. For  Koszul resolutions $F\xra{\simeq} M$ and $G\xra{\simeq} N$,  $$\VV_E(M,N)=\gsupp{\A} (\CC_E(F,G)\otimes_Q k).$$ 
\end{lemma}
\begin{proof}
Let $\iota: \PS_k^{n-1}\inca \PS_Q^{n-1}$ be the canonical embedding induced by the projection $\Poly\to \A$. By definition,  $$\iota(\VV_E(M,N))=\gsupp{\Poly}\left(\Ext_E^*(M,N)\otimes_Q k\right).$$ The result follows from  the equalities \begin{align*}
\gsupp{\Poly}\left(\Ext_E^*(M,N)\otimes_Q k\right)&=\gsupp{\Poly} \left(\Ext_E^*(M,N)\otimes_\Poly \A\right) \\
&=\V_E(M,N)\cap \gsupp{\Poly}\A \\
&=\gsupp{\Poly}\CC_E(F,G)\cap \gsupp{\Poly} \A \\
&=\gsupp{\Poly}(\CC_E(F,G)\ot_\Poly \A )\\
&=\gsupp{\Poly}(\CC_E(F,G)\otimes_Q k)
\end{align*} where the second equality is Proposition  \ref{gradedsup}(4), the third equality is by Proposition \ref{refg}(2), the fourth equality is \ref{tensor} and the last equality follows from Proposition \ref{underivetensor}.  \end{proof}

\begin{chunk}\label{c:cx}Let $M$ and $N$ be a pair of objects in $\D^f(E)$. We set $\EE:=\Ext_E^*(M,N)\otimes_Q k$. 
The \emph{complexity of $M$ and $N$ over $E$}, denoted $\cx_E(M,N)$, measures the polynomial growth of the sequence $\{\dim_k \EE^i\}_{i\in \N}$. More precisely, $$\cx_E(M,N):=\inf\{b\in \N: \dim_k\EE^i\leq ai^{b-1}\text{ for some }a>0\text{ and all }i\gg 0\}, $$ and we  define $\cx_EM:=\cx_E(M,k)$. 
\end{chunk}

\begin{chunk}\label{cx}Let $M$ and $N$ be objects of $\D^f(E)$ such that $\Ext_Q^{\gg 0}(M,N)=0$. By assumption $\EE:=\Ext_E^*(M,N)\otimes_Q k$ is a finitely generated graded $\A$-module. Thus,  by the Hilbert-Serre theorem, $\cx_E(M,N)$ is exactly one more than the dimension of $\VV_E(M,N)$ viewed as  a Zariski closed subset of $\PS_k^{n-1}$.  Thus, $\cx_E(M,N)\leq n$. 
\end{chunk}

\begin{example} \
\begin{enumerate}
\item Let $E$ be a minimal derived complete intersection. By Example \ref{example1}(2) and \ref{cx}, it follows that $\cx_Ek=n.$  This is analogous to, and recovers, the statement for ordinary complete intersection rings; namely, the complexity of the residue field is the codimension of the complete intersection.
\item  Theorem \ref{cmsv} can be translated into the following numerical measure of a complete intersection, $Q/(\f)$ is a complete intersection if and only if $\cx_EQ/(\f)=0$. 
\end{enumerate}
 \end{example}
 
 The same proof  as in \cite{AI2}, using Koszul objects, can be adapted to this setting to establish the following result. We leave this to the reader.
 \begin{proposition}
 Let $E$ be a minimal derived complete intersection. For each closed subset $C$ of $\PS_k^{n-1}$, there exists an object $M$ of $\D^f(E)$ such that $\VV_E(M)=C.\qedhere$
 \end{proposition}

\subsection{Symmetry of Cohomological Support, $\Ext$, $\Tor$ and Complexity}\label{s3}

We now prove some of the major results of the article. The first  relates  the closed subsets $\V_E(M)$, $\V_E(N)$ and $\V_E(M,N)$ of $\mathbb{P}_Q^{n-1}$ when $E$ is a derived complete intersection and contains Theorem \ref{intro1} from the introduction. See the discussion in the introduction, and Remark \ref{m}, for a comparison of the proofs of Theorem \ref{MAIN1} and \cite[Theorem I]{SV} of Avramov and Buchweitz.

\begin{theorem}\label{MAIN1}
Let $E$ be a derived complete intersection  and assume that $M$, $M'$, $N$ and $N'$ are objects of $\D^f(E)$. 
\begin{enumerate}
\item $\V_E(M,N)\cap \V_E(M',N')=\V_E(M,N')\cap \V_E(M',N)$.
\item $\V_E(M,N)=\V_E(M)\cap \V_E(N)=\V_E(N,M).$
\item $\Ext_E^{\gg 0}(M,N)=0$ if and only if $\V_E(M)\cap \V_E(N)=\emptyset$. 
\end{enumerate}
\noindent Furthermore, assume that  $(Q,\n,k)$ is local.
\begin{enumerate}
 \setcounter{enumi}{3}
 \item $\VV_E(M)=\VV_E(M,k)=\VV_E(k,M)$.
\item The following inequalities hold $$\cx_EM+\cx_EN-n\leq \cx_E(M,N)=\cx_E(N,M)\leq \max\{\cx_EM, \cx_E N\}.$$
\end{enumerate}
\end{theorem}
\begin{proof}First, we claim that (2)-(5) all follow from (1). Indeed, 
 \begin{equation}\V_E(M)\cap \V_E(N)=\V_E(M,N)\cap \V_E(N,M)\label{eqsu}\end{equation} by (1). Combining (\ref{eqsu}) with Proposition \ref{bsp}(5) provides us with  $$\V_E(M,N)\con\V_E(M)\cap \V_E(N)=\V_E(M,N)\cap \V_E(N,M)$$ and so $\V_E(M,N)\con \V_E(N,M)$. By symmetry, we obtain $\V_E(N,M)\con \V_E(M,N).$ Thus, $\V_E(M,N)=\V_E(N,M)$ and now   (\ref{eqsu}) finishes the proof that (2) is a consequence of (1). 
Also, (3) follows directly from  (2) and Proposition \ref{refg}(3).  

In the local case, using (2) and \ref{c:v2} the third and fourth equalities below hold \begin{align*}
\VV_E(M)&=\VV_E(M)\cap \PS_k^{n-1} \\
&=\VV_E(M)\cap \VV_E(k) \\
&=\VV_E(M,k)\\
&=\VV_E(k,M).
\end{align*} Hence, (4) holds. Finally, (2) and (4) imply that  $$\VV_E(M,N)=\VV_E(M,k)\cap \VV_E(N,k),$$ and combing this with 
 \ref{cx} shows (5) holds.\footnote{This last deduction is  essentially the same argument from \cite[5.7]{SV} for complexity over a complete intersection ring.}  Hence, as claimed, it  suffices to show (1) holds. 

Now we begin the proof of (1). As $E$ is a derived complete intersection,  there exists finite Koszul resolutions $F\xra{\simeq} M$, $F'\xra{\simeq} M'$, $G\xra{\simeq} N$ and $G'\xra{\simeq} N'$ (c.f. \ref{Koszulres}). First, we claim that there is an isomorphism of DG $\Poly$-modules \begin{equation}\label{importantiso}\CC_E(F,G)\otimes_{\Poly} \CC_E(F',G')\cong\CC_E(F,G')\otimes_{\Poly} \CC_E(F',G).\end{equation}
Indeed,  there is an isomorphism of graded $\Poly$-modules $$(\CC_E(F,G)\otimes_{\Poly} \CC_E(F',G'))^\natural \cong \Poly\otimes_Q (\Hom_Q(F,G)\otimes_Q \Hom_Q (F',G'))^\natural.$$  We let $T(F,G;F',G')$ be the DG $\Poly$-module with $$T(F,G; F',G')^\natural \cong \Poly\otimes_Q (\Hom_Q(F,G)\otimes_Q \Hom_Q (F',G'))^\natural$$ and $$\del^T:=1\otimes \del^{\Hom(F,G)\otimes \Hom(F',G')}+\delta(F,G)+\delta(F',G')$$ where 
$$\delta(F,G):=\sum_{i=1}^n\chi_i \otimes(\Hom(\lambda_i,1)\otimes 1-\Hom(1,\lambda_i)\otimes 1)$$
$$\delta(F',G'):=\sum_{i=1}^n\chi_i\otimes1\otimes (\Hom(\lambda_i,1)- \Hom(1,\lambda_i)).$$ 
Also, we have the following isomorphism of complexes of  $Q$-modules $$
\vp: \Hom_Q(F,G)\otimes_Q \Hom_Q (F',G')\xra{\cong} \Hom_Q(F,G')\otimes_Q \Hom_Q (F',G).$$ Moreover, this induces a DG $\Poly$-module isomorphism $\Psi$ that fits into the following diagram 
$$
\begin{tikzcd}
\CC_E(F,G)\otimes_\Poly\CC_E(F',G')\arrow["\Psi"]{r} \arrow["\cong"]{d}  & \CC_E(F,G')\otimes_{\Poly} \CC_E(F',G) \\
T(F,G;F',G')\arrow["1\otimes \vp"]{r}
\arrow["\cong",swap]{r} & T(F,G'; F,G') \arrow["\cong",swap]{u}
\end{tikzcd}
$$

Next, by Proposition \ref{refg}(2),   $$\V_E(M,N)=\V_E(M,G)=\gsupp{\Poly}\CC_E(F,G).$$ Similarly, \begin{align*}
\V_E(M,N')&=\gsupp{\Poly}\CC_E(F,G') \\
\V_E(M',N)&=\gsupp{\Poly}\CC_E(F',G) \\
\V_E(M',N')&=\gsupp{\Poly}\CC_E(F',G')  \end{align*}
Therefore,\begin{align*}\V_E(M,N)\cap \V_E(M',N')&=\gsupp{\Poly}\CC_E(F,G)\cap \gsupp{\Poly}\CC_E(F',G') \\ 
&=\gsupp{\Poly}\left(\CC_E(F,G)\ot_{\Poly}\CC_E(F',G')\right) \\
&=\gsupp{\Poly}\left(\CC_E(F,G)\otimes_{\Poly}\CC_E(F',G')\right) \end{align*} where the second equality uses \ref{tensor} and the third equality uses Proposition \ref{underivetensor}. 
Similarly, $$\V_E(M,N')\cap \V_E(M',N)=\gsupp{\Poly}\left(\CC_E(F,G')\otimes_{\Poly}\CC_E(F',G)\right)$$
Finally, by (\ref{importantiso}), we get the desired result. 
\end{proof}

\begin{remark}\label{m}
The fact that the eventual vanishing of Ext-modules over a complete intersection is symmetric in the module arguments of Ext has already been established  in each of the following \cite{SV,AI3,BW,HJ}. Theorem \ref{MAIN1} adds a new proof to the collection that incorporates ideas from \cite{SV} and \cite{BW}. Moreover, Theorem \ref{MAIN1} recovers the stronger result, of Avramov and Buchweitz \cite{SV}, that complexity is in fact symmetric in both module arguments over a complete intersection. We end this section with one last result that is directly inspired by \cite[Theorem III]{SV}. The argument for ``(1) implies (3)" in Theorem \ref{MAIN2} differs from the one in \cite[Theorem III]{SV} as we do not resort to a theory of complete resolutions, Tate cohomology, and bands of vanishing over a complete intersection.   
\end{remark}

\begin{theorem}\label{MAIN2}
Let $E$ be a  derived complete intersection. For a pair of objects $M$ and $N$ of $\D^f(E)$, the following are equivalent: 
\begin{enumerate}
\item   $\Ext_E^{\gg 0}(M,N)=0$.
\item   $\Ext_E^{\gg 0}(N,M)=0$. 
 \item    $\Tor^E_{\gg 0}(M,N)=0$. 
\end{enumerate}
\end{theorem}
\begin{proof}
The equivalence of (1) and (2) is a consequence of Theorem \ref{MAIN1}(3). We let $(-)^\dagger:=\RHom_E(-,\Hom_Q(E,Q))$.

$(3)\implies (1)$: By assumption $M\ot_E N$ is an object of $\D^f(E)$. Since $Q$ is Gorenstein, then   \ref{kosgor} implies that   $(M\ot_E N)^\dagger$ is an object of  $\D^f(E)$. Using adjunction, $$(M\ot_E N)^\dagger\simeq \RHom_E(M,N^\dagger).$$   Therefore, $\Ext_E^{\gg 0}(M,N^\dagger)=0$ and so   $$\V_E(M)\cap \V_E(N^\dagger)=\emptyset.$$  Thus,  Proposition \ref{dual}(2) yields $$V_E(M)\cap \V_E(N)=\emptyset.$$ That is, $\Ext_E^{\gg 0}(M,N)=0$ (see Theorem \ref{MAIN1}(3)).

$(1)\implies (3)$:  Assume that $\Ext_E^{\gg 0}(M,N)=0$. By   Theorem \ref{MAIN1}(3) and Proposition \ref{dual}(2), $$\V_E(M)\cap \V_E(N^\dagger)=\emptyset.$$  That is, $\RHom_E(M,N^\dagger)$ is an object of $\D^f(E)$. 
Using adjunction, we obtain an isomorphism  $$(M\ot_EN)^\dagger\simeq \RHom_E(M,N^\dagger)$$ and so $(M\ot_EN)^\dagger$ is an object of  $ \D^f(E).$  Since $(M\ot_E N)^\dagger$ is an object of $\D^f(E)$ and again using that $Q$ is Gorenstein, it follows that $(M\ot_EN)^{\dagger\dagger}$ is an object of $\D^f(E)$ (by \ref{kosgor}). Finally, applying Proposition \ref{p:duality}$$M\ot_EN\simeq(M\ot_EN)^{\dagger\dagger},$$  and hence $M\ot_E N$ is an object of  $\D^f(E)$. Therefore, $\Tor_E^{\gg 0}(M,N)=0$ as needed. 
\end{proof}



\section{Cohomological Support for Local Rings}\label{localvar}

\subsection{Definition and  Properties}

\begin{chunk}\label{defdc}
Let $(R,\m,k)$ be a commutative noetherian local ring. We let $\widehat{R}$ be the $\m$-adic completion of $R$. We say that $E=\Kos^Q(\f)$ is a \emph{(minimal) derived complete intersection approximation}, or \emph{(minimal) DCI-approximation}, provided that  $E$ is a (minimal) derived complete intersection with $Q$ complete and $\h_0(E)=\widehat{R}$. By the Cohen Structure Theorem, DCI-approximations of $R$ exist and are unique up quasi-isomorphism (see \ref{kosfunc}).  Moreover, if $E=\Kos^Q(f_1,\ldots,f_n)$ is a minimal DCI-approximation of $R$ then $n$ is independent of choice of $Q$ and $\f$; namely, $n=\dim_k \h_1(K^R)$ where $K^R$ is the Koszul complex on a minimal generating set for $\m$.  We call $n$ the \emph{derived codimension} of $R$. 
\end{chunk}

\begin{theorem}\label{wdvar}
Let $(R,\m,k)$ be a local ring. If $E=\Kos^Q(f_1,\ldots, f_n)$ and $E'=\Kos^{Q'}(f_1',\ldots, f_m')$ are DCI approximations of $R$, then there exists a  diagram of schemes $$\PS_Q^{n-1}\xra{\alpha} \PS_P^N\xra{\beta} \PS_P^N\xla{\alpha'} \PS_{Q'}^{m-1}$$
where  $P$
 is a regular local ring, $\alpha$ and $\alpha'$ are embeddings and $\beta$ is a linear automorphism such that 
 $$\beta(\alpha(\V_E(M.N)))=\alpha'(\V_{E'}(M,N))\text{ in }\PS_{P}^N$$ for any pair of objects $M$ and $N$ in  $\D^f(\widehat{R})$.
\end{theorem}
\begin{proof}
We write \begin{align*}E&=Q\langle \xi_1,\ldots, \xi_n|\del \xi_i=f_i\rangle,\\
E'&= Q'\langle \xi_1',\ldots, \xi_m'|\del \xi_i'=f_i'\rangle
\end{align*} then by \ref{kosfunc}, there exists surjections of regular local rings $\pi: P\to Q$ and $\pi': P'\to Q'$. Let 
  $$A:=P\langle \xi_1,\ldots \xi_n, \tau_{n+1},\ldots , \tau_{N+1}|\del \xi_i=g_i\text{ and }\del \tau_{n+i}=x_i\rangle$$ where  $\pi(g_i)=f_i$ and $\x$ minimally generates $\ker\pi$. Similarly, we let    $$A':=P\langle \xi_1',\ldots \xi_m', \tau_{m+1}',\ldots , \tau_{N+1}'|\del \xi_i'=g_i'\text{ and }\del \tau_{m+i}'=x_i'\rangle$$ where  $\pi'(g_i')=f_i'$ and $\x'$ minimally generates $\ker\pi'$.

  First,  $\pi$ induces a canonical surjection $$P[\chi_1,\ldots, \chi_n, \zeta_{n+1},\ldots, \zeta_{N+1}]\to Q[\chi_1,\ldots, \chi_n],$$  which defines an embedding $\alpha: \PS_Q^{n-1}\to \PS_P^N$.  
  Let $\e: A\xra{\simeq}E$ be the surjective quasi-isomorphism induced by $\pi$ mapping $\xi_i\mapsto \xi_i$ and $\tau_i\mapsto 0$. Applying Proposition \ref{compare} to $\e$ yields that  $$\V_A(M,N)=\alpha(\V_E(M,N)).$$
  Similarly, $\pi'$ induces a canonical surjection $$P[\chi_1',\ldots, \chi_m', \zeta_{m+1}',\ldots, \zeta_{N+1}']\to Q'[\chi_1',\ldots, \chi_m'],$$  which defines an embedding $\alpha: \PS_Q'^{m-1}\to \PS_P^N$ such that $$\V_{A'}(M,N)=\alpha'(\V_{E'}(M,N)).$$ Finally, since $P$ is local and $\x,\g$ and $\x', \g'$ minimally generate the same ideal in $P$ there is an isomorphism of DG $P$-algebras $\vp: A\xra{\simeq }A'$. By Proposition \ref{compare}, this determines a linear automorphism $\beta:\PS_P^N\to \PS_P^N$ such that $$\beta(\V_A(M,N))=\V_{A'}(M,N).\qedhere$$
   \end{proof}

\begin{chunk}
Let $(R,\m)$ be a commutative noetherian local ring and let $\widehat{R}$ denote its $\m$-adic completion. Let $E:=\Kos^Q(f_1,\ldots,f_n)$ be a minimal DCI-approximation of $R$. For each pair of objects $M$ and $N$ of $\D^f(R)$, $\V_E(M\otimes_R \widehat{R},N\otimes_R \widehat{R})$ is a closed subset of $\mathbb{P}_Q^{n-1}$. For $E':=\Kos^{Q'}(f_1',\ldots, f_n')$ another minimal DCI-approximation of $R$, then $\V_{E'}(M\otimes_R \widehat{R},N\otimes_R \widehat{R})$ is a closed subset of $\mathbb{P}_{Q'}^{n-1}$. By Theorem \ref{wdvar}, $\V_E(M\otimes_R \widehat{R},N\otimes_R \widehat{R})$  and $\V_{E'}(M\otimes_R \widehat{R},N\otimes_R \widehat{R})$ determine the same closed subset of $\PS_P^N$ where $P$ is a regular local ring  and the following is a commutative diagram of surjective ring morphisms $$
\begin{tikzcd}
  & P\arrow{rd} \arrow{ld} &   \\
 Q \arrow{rd}&  & Q' \arrow{ld}\\
 & \widehat{R} &   
\end{tikzcd}$$
\end{chunk}

\begin{remark}
Since $\V_R(-,-)$ is defined in terms of $\V_E(-,-)$ where $E$ is a minimal DCI-approximation of $R$, all of the results regarding the support sets   from Sections \ref{s1}, \ref{s2}, and \ref{s3} hold for $\V_R(-,-),$ as well. Again, to save space, we do not list them all but instead list the main ones in Theorem \ref{thmls}, below. 
\end{remark}

\begin{definition}\label{localdef}
Let $(R,\m)$ be a commutative noetherian local ring and let $\widehat{R}$ denote its $\m$-adic completion. For a pair of objects $M$ and $N$ of $\D^f(R)$, we define the \emph{cohomological support  of the pair $(M,N)$} to be $$\V_R(M,N):=\V_E(M\otimes_R\widehat{R}, N\otimes_R\widehat{R})$$ where $E$ is a minimal DCI-approximation of $R$. We define the \emph{cohomological support  of $M$} to be $\V_R(M):=\V_E(M,M).$ 
\end{definition}

\begin{theorem}\label{thmls}
Let $(R,\m,k)$ be a commutative noetherian local ring. 
\begin{enumerate}
\item $\V_R(M,N)\cap\V_R(M',N')=\V_R(M,N')\cap \V_R(M',N)$ for all objects  $M,M',N,N'$ in $\D^f(R)$.
\item $\V_R(M,N)=\V_R(N,M)$ for all $M,N\in \D^f(R)$.
\item 
The following are equivalent:
\begin{enumerate}
\item $R$ is a complete intersection.
\item $\V_R(R,k)=\emptyset$.
\item $\V_R(R)=\emptyset$. 
\item $\V_R(M,k)=\emptyset$ for some nonzero finitely generated $R$-module $M$.
\item $\V_R(M)=\emptyset$ for some nonzero finitely generated $R$-module $M$. 
\end{enumerate}
\end{enumerate}
\end{theorem}
\begin{proof}
The statements (1) and (2) follow immediately from Theorem \ref{MAIN1}. For (3), 
let $E=\Kos^Q(f_1,\ldots,f_n)$ is a minimal DCI-approximation of $R$ and let $\Poly$ be $Q[\chi_1,\ldots,\chi_n]$, the ring of cohomology operators for $E$. 
Most of (3) follows from Theorem \ref{cmsv}. However, it remains to prove the following claim.

\noindent\textbf{Claim}: For an object $M$ of $\D^f(E)$,  $\V_E(M)=\emptyset$ if and only if  $\V_E(M,k)=\emptyset$. 

Indeed,  $\V_R(M,k)=\emptyset$ if and only if $M$ is an object of $\Thick_{\D(E)}E$. Thus, $\Ext_E^{\gg 0}(M,M)=0$ and hence Proposition \ref{refg}(3), it follows that $\V_E(M)=\emptyset$. 
 
 Now assume that $\V_E(M)=\emptyset$. By Proposition \ref{refg}(3), $\Ext_E^{\gg 0}(M,M)=0$ and hence,  $$\left(\Ext_E^*(M,M)\otimes_\Poly (\Poly\otimes_Qk)\right)^{\gg 0}=\Ext_E^{\gg 0}(M,M)\otimes_Qk=0.$$ Thus, \begin{align*}\emptyset &=\gsupp{\Poly}(\Ext_E^{*}(M,M)\otimes_\Poly (\Poly\otimes_Qk)) \\
 &=\gsupp{\Poly}(\Ext_E^*(M,M))\cap \gsupp{\Poly}(\Poly\otimes_Q k) \\
 &=\V_E(M)\cap \V_E(k) \\
 &=\V_E(M,k)
 \end{align*}  where  the second equality holds since $\Ext_E^*(M,M)$ and $S\otimes_Q k$ are finitely generated graded $\Poly$-modules, the third equality holds by  Example \ref{example1}(2)  and Theorem \ref{MAIN1}(2) yields the last equality.
 \end{proof}



\subsection{Relation to Other Supports}

Let $Q$ be a commutative noetherian ring and set $E:=\Kos^Q(\f)$ where $\f=f_1,\ldots,f_n$ is a list of elements from $Q$. We set $\Poly:=Q[\chi_1,\ldots,\chi_n]$, the ring of cohomology operators associated to $E$. In this section,  we set $R:=\h_0(E)=Q/(\f)$.

\begin{chunk}\label{regseq}Assume that $\f$ is a $Q$-regular sequence. 
 For each pair of DG $R$-modules, $\Ext_R^*(M,N)$ is a graded $\Poly$-module.  
  By \cite[2.4]{CD2} and \ref{augiso}, we have that the quasi-isomorphism $\e:E\xra{\simeq} R$ induces an isomorphism of graded $\Poly$-modules $$\Ext_{R}^*(M,N)\xra{\Ext_\e(M,N)} \Ext_E^*(M,N).$$ 
 \end{chunk}

\begin{chunk}[Burke-Walker]\label{BW}Assume that $\f$ is a $Q$-regular sequence. 
Let $M$ and $N$ be a  pair of objects from $\D^f(R)$ and assume that $M$ is in $\Thick_{\D(Q)}Q$ when regarded as a complex of $Q$-modules via the canonical projection $Q\to R$. In \cite{BW}, Burke and Walker define the \emph{cohomological support of $(M,N)$}, denoted $\V_Q^\f(M,N)$,  which in the notation of that paper, satisfies $$\Var_Q^\f(M,N)=\gsupp{\Poly\otimes_Q R}\Ext_R^*(M,N).$$ That is, $$\Var_Q^\f(M,N)=\V_E(M,N)\times_{\Spec Q} \Spec R.$$Thus, starting from $\V_E(M,N)$ one can obtain the varieties defined by Avramov and Buchweitz \cite{SV}, Benson, Iyengar and Krause \cite{BIK},  and Stevenson \cite{Stevenson} when $\f$ is a $Q$-regular sequence (see  \cite[8.1]{BW} for more details).
\end{chunk}

\begin{chunk}[D. Jorgensen]\label{relateJ}
Let $(Q,\n,k)$ be local. For simplicity, we  assume that $k$ is algebraically closed. In \cite{Jor}, D. Jorgensen defines the cohomological support for  a pair of finitely generated $R$-modules to be $$\Var({Q},\f;M,N):=\{(a_1,\ldots, a_c)\in \mathbb{A}_{{k}}^n: \Ext_{{Q}_{\bm{a}}}^*({M},{N})\text{ is unbounded}\}\cup\{\bm{0}\}$$ where  $${Q}_{\bm{a}}={Q}/(\tilde{a}_1 f_1+\ldots+ \tilde{a}_nf_n)$$ and $\tilde{a}_i$ is a lifting of $a_i$ to ${Q}$.  By \cite[2.1]{Jor}, if ${Q}$ is a domain and $(\f)\con \n$, then $\Var({Q},\f;M,N)$ is  a well-defined, (Zariski) closed  subset of $\mathbb{A}_{{k}}^n.$  By \cite[2.5]{SV} and \ref{BW}, when $\f$ is a regular sequence and $\Ext_Q^{\gg 0}(M,N)=0$, 
$$\Var({Q},\f;M,N)=\nu^{-1}\max(\VV_E(M,N)))$$
 where $\nu: \mathbb{A}_{{k}}^n\backslash\{\bm{0}\}\to \PS_{{k}}^{n-1}$ is given by $$(a_1,\ldots, a_n)\mapsto (a_i\chi_j-a_j\chi_i).$$  In Proposition \ref{unstable}, below, we establish a more general relation between $\V_E(M,N)$ and $\Var({Q},\f;M,N)$ for finitely generated $R$-modules. 
\end{chunk}

\begin{theorem}
\label{unstable}Assume  $(Q,\n,k)$ is  local  and $(\f)\con\n$ contains a regular element.  Let $\nu: \mathbb{A}_k^n\backslash\{\bm{0}\}\to \PS_k^{n-1}$ given by $$(a_1,\ldots, a_n)\mapsto (a_i\chi_j-a_j\chi_i).$$ 
\begin{enumerate}
\item Let $M$ and $N$  be a pair of objects from  $\D^f(E)$ such that  $\Ext_Q^{\gg 0}(M,N)=0$. For  $\bm{a}=(a_1,\ldots, a_n)\in \mathbb{A}_k^n\backslash\{\bm{0}\}$,  $\nu(\bm{a})\in \VV_E(M,N)$ if and only if  $\Ext_{E_{\bm{a}}}^*(M,N)$ is unbounded where $$E_{\bm{a}}:=\Kos^Q(\tilde{a}_1f_1+\ldots+ \tilde{a}_nf_n)$$ and $\tilde{a}_i$ is a lifting of $a_i$ to $Q$. 
\item Assume that  $k$ is algebraically closed. Then 
 $$\Var({Q},\f;M,N)=\nu^{-1}(\max\VV_E(M,N))$$  for  each pair of objects $M$ and $N$ of $\D^f(R)$ with  $\Ext_Q^{\gg 0}(M,N)=0$. 
 \end{enumerate}
 \end{theorem}
 \begin{remark}\
 \begin{enumerate}
 \item 
 The proof of  Theorem \ref{unstable}(1) is essentially the same as the proof  \cite[2.5]{SV}. Instead of importing that proof with the slight, necessary modifications, we offer a new proof of Proposition \ref{unstable}(1) that holds when $Q$ is regular. The difference is that the proof below works at the chain level rather than after taking homology. 
   \item Since $(\f)\con \n$ contains a regular element,  for each $a\in \mathbb{A}_k^{n}\backslash \{\bm{0}\}$, $E_{\bm{a}}\xra{\simeq} Q_{\bm{a}}$.  Hence, for a pair of finitely generated $R$-modules $M$ and $N$,   $\Ext_{E_{\bm{a}}}^*(M,N)$ is unbounded if and only if $\Ext_{Q_{\bm{a}}}^i(M,N)\neq 0$ for infinitely many $i$.
Hence,   Theorem \ref{unstable}(2)  follows from \ref{unstable}(1) and the graded Nullstellensatz. \end{enumerate} 
 \end{remark}
\begin{proof}[Proof of Theorem \ref{unstable}(1) when $Q$ is regular.]

By Proposition \ref{compare}, we can assume that $\bm{a}=(1,0,\ldots, 0)$. Thus, $$\nu(\bm{a})=(\chi_2,\ldots, \chi_n)\in \Proj \A.$$  By assumption and Proposition \ref{refg}(2), $$\V_E(M,N)=\gsupp{\Poly} \CC_E(F,G)$$ where $F\xra{\simeq} M$ and $G\xra{\simeq} N$ are  Koszul resolutions.  
By Lemma \ref{lt},  $$\VV_E(M,N)=\gsupp{\A}(\CC_E(F,G)\otimes_Q k).$$ 
Also, we have isomorphisms of DG $\A$-modules \begin{align*}( \CC_E(F,G)\otimes_Qk )\ot_\A\kappa_\A(\nu(\bm{a})) & =\CC_E(F,G)\otimes_Qk \otimes_\A\kappa_\A(\nu(\bm{a})) \\
&\cong(\CC_E(F,G)\otimes_\Poly \Poly/(\chi_2,\ldots, \chi_n)\otimes_Qk)_{\chi_1} \\
&\cong(\CC_{E_{\bm{a}}}(F,G)\otimes_Q k)_{\chi_1} \end{align*} where first equality holds by Proposition \ref{underivetensor} and  the third equality holds since  $F$ and $G$ are Koszul resolutions of $M$ and $N$, respectively, over $E_{\bm{a}}$. 

 In summary, $\nu(\bm{a})\in \VV_E(M,N)$ if and only if $\CC_{E_{\bm{a}}}(F,G)_{\chi_1}\not\simeq 0$. Since localization is exact, the latter is equivalent to\begin{equation}\Ext_{E_{\bm{a}}}^*(M,N)_{\chi_1}\neq 0.\label{e:9}\end{equation} As $\Ext_{E_{\bm{a}}}^*(M,N)$ is a finitely generated graded $Q[\chi_1]$-module satisfying (\ref{e:9}),  it follows, equivalently, that $\Ext_{E_{\bm{a}}}^*(M,N)$ is unbounded.  Now using the isomorphism $$\Ext_{E_{\bm{a}}}^*(M,N)\cong \Ext_{Q_{\bm{a}}}^*(M,N),$$ we obtain that $\nu(\bm{a})\in \VV_E(M,N)$ if and only if  $\Ext_{Q_{\bm{a}}}^i(M,N)$ is nonzero for infinitely many values of $i$.  
\end{proof}



\subsection{A Study of  $\V(R)$ for Local Rings with Small Codepth}\label{cod}

Let $(R,\m,k)$ be a local commutative noetherian ring. 
In this section we investigate the  \emph{cohomological support  of $R$}; namely, we study   $$\V(R):=\V_R(R,k).$$ By Theorem \ref{thmls}, $\V(R)$ is empty exactly when $R$ is a complete intersection. Hence, of particular is interest is describing $\V(R)$  when $R$ is not a complete intersection. 

For the rest of the section we fix the following notation. 
\begin{Notation}\label{n4}
Let $E:=Q\langle \xi_1,\ldots, \xi_n|\del \xi_i=f_i\rangle$ be a minimal DCI-approximation of $R$. In particular, $Q$ is a regular local ring; we let $\n$ denote the maximal ideal of $Q$ and note that $Q/\n\cong k$.  Fix a minimal free resolution $F\xra{\simeq}\widehat{R}$ over $Q$.  For a DG $E$-module $X$, we let $\lambda_i$ denote left multiplication by $\xi_i$.  Finally, note that  $\V(R)$ is a subset of $\mathbb{P}_k^{n-1}$ where $n$ is the derived codimension of $\widehat{R}$ (c.f. \ref{defdc}).\end{Notation}

\begin{chunk} 
\label{resalg}
Suppose $\pd_Q \widehat{R}\leq 3$. By \cite{BE}, $F$ admits a DG $Q$-algebra structure. As $\h_0(E)=\widehat{R}$, it follows that $F$ inherits a DG $E$-algebra structure. 
By Remark \ref{complex}, $\CC_E(F,k)$ is the following complex of graded $\A$-modules
 \begin{equation}\label{Ares}\ldots\to \shift^{-4}\A\otimes_k \ov{F_2}\xrightarrow{\del_2} \shift^{-2}\A\otimes_k \ov{F_1}\xrightarrow{\del_1} \A\otimes_k\ov{F_0}\to 0\end{equation}  where  $\ov{F_i}=\Hom_Q(F_i,k)$ and $$\del=\sum_{i=1}^n \chi_i\otimes \Hom(\lambda_i, k).$$
\end{chunk}

We define the \emph{codepth of R} to be $$\codepth R:=\dim_k(\m/\m^2)-\depth R.$$ By the Auslander-Buchsbaum formula $\codepth R=\pd_Q \widehat{R}.$
\begin{example}\label{ex3}
Assume $R$ is not a complete intersection, the derived codimension of $R$ is $n$, and $R$ satisfies one of the following conditions:
\begin{enumerate}
\item  $\codepth R=2$, or
\item  $\codepth R=3$ and $R$ is Gorenstein.
\end{enumerate} Then $\V(R)=\mathbb{P}_k^{n-1}$.

 Indeed, in either case $F_1F_1\con \n F_2$ (see \cite[2.1.2]{IFR} and \cite[2.1.3]{IFR}, respectively) where we adopt the notation from \ref{resalg} .  Hence,  
(\ref{Ares}) has the following form $$
0 \to \shift^{-6}\A\otimes_k\ov{F_3}\xra{\del_3} \shift^{-4}\A\otimes_k \ov{F_2}\xrightarrow{0} \shift^{-2}\A\otimes_k \ov{F_1}\xrightarrow{\begin{pmatrix}\chi_1& \ldots & \chi_n\end{pmatrix}} \A\otimes_k\ov{F_0}\to 0.$$ Therefore, in either case, $\shift^{-2}\A$ is a submodule of $\Ext_E^*(\widehat{R},k)$ which justifies the claim. 
\end{example}

More generally, when the codepth of the ring is small, we have a complete description of $\V(R)$.
We quickly recap the DG algebra structure on $F$ when $\codepth R=3$ (see \cite{AKM} for more details). 

\begin{chunk}\label{strcodepth3}
Assume $\codepth R=3$; that is, $\pd_Q \widehat{R}=3$. 
We fix the following notation: \[F_1=\bigoplus_{i=1}^n Q a_i,  \  \ 
F_2=\bigoplus_{i=1}^m Q b_i, \  \ 
F_3=\bigoplus_{i=1}^\ell Q c_i.\]
The  DG $E$-module structure on $F$ is determined by $$\xi_i\cdot x:=a_i\cdot x$$ (see \ref{kosaction}). Therefore,  (\ref{Ares}) is the following complex of graded $\A$-modules:
 \begin{equation}\label{Ares2}
 0\to\shift^{-6} \A\otimes_k \oplus k\gamma_i \xra{\del_3}\shift^{-4}\A\otimes_k\oplus k\beta_i\xrightarrow{\del_2} \shift^{-2}\A\otimes_k \oplus k\alpha_i\xrightarrow{\del_1} \A\otimes_kk\to 0
 \end{equation}
where $$\alpha_i=\Hom(a_i,k), \ \beta_i=\Hom(b_i,k), \ \gamma_i=\Hom(c_i,k)$$ and 
$$\del_1=\begin{pmatrix} \chi_1 & \ldots &\chi_n\end{pmatrix}.$$

In  \cite{AKM}, Avramov, Kustin and Miller showed that there are five classes of algebra structures on $F$ modulo $\n$: {\bf CI}, {\bf TE}, {\bf B}, {\bf G}(r), {\bf H}$(p,q)$. We summarize the relevant results from \cite{AKM} below:  
\begin{enumerate}
\item $R$ belongs to \textbf{CI} if and only if  $R$ is a complete intersection.   
\item When $R$ belongs to  \textbf{G$(r)$}, there exists $r\geq 2$ such that  $a_ib_i=c_1$ modulo $\n$ for all $1\leq i\leq r$ and all other products on $F$ are zero modulo $\n$.

\item If  $R$ belongs to \textbf{TE}, $$a_2a_3=b_1, \ a_3a_1=b_2, \ a_1a_2=b_3$$ and all other products in $F$ are zero modulo $\n$.

\item Assume $R$ is in \textbf{B}. In this case, we have the following equations holding modulo $\n$: $$a_1a_2=b_3, a_1b_1=c_1, a_2b_2=c_1,$$ and all other products of basis elements of $F$ are zero.

\item For $R$  in \textbf{H$(p,q)$},   $p< n$,  $q\leq \ell$, and modulo $\n$   $$a_{p+1}a_i=b_i \text{ for all }1\leq i \leq p, a_{p+1}b_{p+i}=c_i\text{ for all }1\leq i\leq q,$$ and all other products of basis elements of $F$ are zero.
\end{enumerate}
\end{chunk}

\begin{theorem}\label{ethm}
We adopt the assumptions and  notation from \ref{strcodepth3}, and set $\EE:=\Ext_E^*(\widehat{R},k)$ and $\mathcal{Z}:=\ker \del_1.$
\begin{enumerate}
\item  When $R$ belongs to \emph{\textbf{CI}}, $\EE=k$. 

\item When $R$ belongs to \emph{\textbf{G$(r)$}}, $$\EE\cong k\oplus \mathcal{Z}\oplus \shift^{-4} \frac{\oplus_{i=1}^m \A \beta_i}{(\sum_{i=1}^r\chi_i\beta_i)}\oplus \shift^{-6} \A^{\ell-1}.$$

\item When $R$ belongs to \emph{\textbf{TE}}, $$\EE\cong k\oplus \frac{\mathcal{Z}}{\shift^{-2}\A(\chi_j\alpha_i-\chi_i\alpha_j)_{1\leq i<j\leq 3}}\oplus \shift^{-4} \A^{m-2}\oplus \shift^{-6} \A^{\ell}.$$

\item When $R$ belongs to \emph{\textbf{B}}, $$\EE\cong k\oplus\frac{\mathcal{Z}}{ \shift^{-2}\A(\chi_2\alpha_1-\chi_1\alpha_2)}\oplus \shift^{-4} \frac{\oplus_{i=1}^m\A\beta_i}{\A(\chi_1\beta_1+\chi_2\beta_2)}\oplus \shift^{-6} \A^{\ell-1}.$$

\item When $R$ belongs to \emph{\textbf{H$(p,q)$}}, $$\EE\cong k\oplus 
\frac{\mathcal{Z}}{\shift^{-2}\A(\chi_{p+1}\alpha_i-\chi_i\alpha_{p+1})_{1\leq i\leq p}}
\oplus \shift^{-4} \frac{\A^{m-p}}{(\chi_{p+1})}\oplus \shift^{-6} \A^{\ell-q}.$$

\end{enumerate}
\end{theorem}
\begin{proof}
\begin{enumerate}
\item This is immediate. 

For each the following cases, recall that $\EE$ can be calculated as the cohomology of (\ref{Ares2}). Hence, we need only specify $\del_2$ and $\del_3$ in each of the remaining cases.

\item Using \ref{strcodepth3}(2), it follows that $$\chi_i\otimes \Hom(\lambda_i,k)(1\otimes\gamma_1)=\chi_i\otimes \beta_i$$ for all $1\leq i \leq r$ and $$\chi_i\otimes \Hom(\lambda_i,k)(1\otimes \beta_i)=0$$ for all $i$. Thus, $$\del_3=\begin{pmatrix}
\chi_1& 0 & \ldots & 0 \\ 
\vdots & \vdots &  & \vdots \\
\chi_r & 0 & \ldots & 0 \\
0 & 0 & \ldots & 0\\
\vdots & \vdots & & \vdots \\ 
0 & 0 & \ldots & 0
\end{pmatrix} \text{ and } \  \del_2=\textbf{0}_{n\times m}$$ where $\textbf{0}_{n\times m}$ is the $n\times m$ matrix consisting of all zeros.

\item Using \ref{strcodepth3}(3),  we get
\begin{align*}
\chi_2\otimes \Hom(\lambda_2,k)(1 \otimes \beta_1)&=\chi_2\otimes \alpha_3 \\
\chi_3\otimes \Hom(\lambda_3,k)(1 \otimes \beta_1)&=-\chi_3\otimes \alpha_2 \\
\chi_3\otimes \Hom(\lambda_3,k)(1 \otimes \beta_2)&=\chi_3\otimes \alpha_1 \\
\chi_1\otimes \Hom(\lambda_1,k)(1 \otimes \beta_2)&=-\chi_1\otimes \alpha_3 \\
\chi_1\otimes \Hom(\lambda_1,k)(1 \otimes \beta_3)&=\chi_1\otimes \alpha_2 \\
\chi_2\otimes \Hom(\lambda_2,k)(1 \otimes \beta_3)&=-\chi_2\otimes \alpha_1 
\end{align*}
and the  values of  $\chi_i\otimes \Hom(\lambda_i,k)$ on  the remaining $1\otimes \gamma_j$ and $1\otimes \beta_j$  are all zero. 
Hence, $\del_3=\textbf{0}_{m\times \ell}$ and
 $$\del_2=\begin{pmatrix}
0 & \chi_3 & -\chi_2 & 0 & \ldots & 0 \\
-\chi_3 & 0 & \chi_1 & 0 & \ldots & 0 \\
\chi_2 & -\chi_1& 0 & 0 & \ldots & 0\\
0 & 0 & 0& 0& \ldots & 0 \\
\vdots & \vdots & \vdots & \vdots & & 0\\
0 & 0 & 0& 0& \ldots & 0 
\end{pmatrix}.
$$

\item Using \ref{strcodepth3}(4), we get the following:
\begin{align*}
\chi_1\otimes\Hom(\lambda_1,k)(1\otimes \gamma_1)&=\chi_1\otimes\beta_1 \\
   \chi_2 \otimes\Hom(\lambda_2,k)(1\otimes \beta_1)&=\chi_2\otimes \beta_2 \\
 \chi_1\otimes \Hom(\lambda_1,k)(1\otimes \beta_3)&=\chi_1\otimes \alpha_2 \\
\chi_2\otimes \Hom(\lambda_2,k)(1\otimes \beta_3)&=-\chi_2\otimes\alpha_1  
\end{align*} and $\chi_i\otimes \Hom(\lambda_i,k)$ vanishes on all the remaining $1\otimes \gamma_j$ and $1\otimes \beta_j$.  In particular, 
$$\del_3=\begin{pmatrix}  \chi_1 & 0 &\ldots & 0 \\
\chi_2 &  0 & \ldots & 0 \\
0 & 0 & \ldots & 0\\
\vdots & \vdots &  & 0 \\
0 & 0 & \ldots & 0 
\end{pmatrix}\text{ and } \ 
\del_2=\begin{pmatrix} 
 0 & 0 & -\chi_2  & 0 & \ldots & 0\\
 0 & 0 & \chi_1 & 0 & \ldots & 0 \\
 0 & 0 & 0 & 0 & \ldots & 0 \\
 \vdots & \vdots & \vdots & \vdots & & \vdots \\
  0 & 0 & 0 & 0 & \ldots & 0 \\
\end{pmatrix}
$$

\item Using \ref{strcodepth3}(5),  we get 
\begin{align*}
\chi_{p+1}\otimes \Hom(\lambda_{p+1},k)(1\otimes \gamma_i)&=\chi_{p+1}\otimes \beta_{p+i} \\
\chi_j\otimes \Hom(\lambda_{j},k)(1\otimes \beta_j)&=-\chi_{j}\otimes \alpha_{p+1} \\
\chi_{p+1}\otimes \Hom(\lambda_{p+1},k)(1\otimes \beta_j)&=\chi_{p+1}\otimes \alpha_{j} 
\end{align*} for $1\leq i \leq q$ and $1\leq j\leq p$, and the  values of  $\chi_i\otimes \Hom(\lambda_i,k)$ on  the remaining $1\otimes \gamma_j$ and $1\otimes \beta_j$  are all zero.    Hence, 
$$\del_3=\begin{pmatrix} \0_{p\times q} & \0_{p\times (\ell-q)} \\
\chi_{p+1} I_q& \0_{q\times (\ell-q)} \\
\0_{(m-p-q)\times q} & \0_{(m-p-q)\times (\ell-q)}
\end{pmatrix}$$
$$
\del_2=\begin{pmatrix}
\chi_{p+1} I_p& \0_{p\times (m-p)} \\
-\bm{\chi} & \0_{1\times (m-p)} \\
\0_{(n-p-1)\times p} & \0_{(n-p-1)\times (m-p)}
\end{pmatrix}$$ where $-\bm{\chi}=\begin{pmatrix} -\chi_1 & \ldots & -\chi_p\end{pmatrix}$ and $I_t$ denotes the $t\times t$ identity matrix. 
\end{enumerate}
\end{proof}

With the notation set in Notation \ref{n4}, we say that $Q\to \widehat{R}$ admits an embedded deformation if there exists $g_1,\ldots, g_n\in \n^2$ with $g_n$ being regular on $Q/(g_1,\ldots, g_{n-1})$ and $$\widehat{R}\cong Q/(g_1,\ldots, g_n).$$ 
We restate Theorem \ref{clsv1} from the introduction for the ease of the reader. 
\begin{theorem}\label{clsv}
Let $(R,\m,k)$  be a commutative noetherian local ring of  derived codimension $n$ and  $\codepth R\leq 3$. The following characterizes the possible subsets of $\mathbb{P}_k^{n-1}$ that $\V(R)$ can realize: 
\begin{enumerate}
\item If $R$ is a complete intersection, then $\V(R)=\emptyset$.
\item If $R$ is not a complete intersection and  $Q\to\widehat{R}$ admits an embedded deformation, then $\V(R)$ is a hyperplane in $\mathbb{P}_k^{n-1}$. \item Otherwise, $\V(R)=\mathbb{P}_k^{n-1}$.
\end{enumerate}
\end{theorem}
\begin{proof}  When $R$ is a complete intersection, $\V(R)=\gsupp{\A}(k)=\emptyset.$ Therefore, we assume that $R$ is not a complete intersection for the rest of the proof (and in particular, $\codepth R$ is $2$ or $3$). 
For $R$ a non-complete intersection with $\codepth R=2$ we appeal to Example \ref{ex3}(1). 
Thus, we assume that $R$ is not a  complete intersection and $\codepth R=3$.

As $R$ is not a complete intersection,  $Q\to \widehat{R}$ admits an embedded deformation if and only if $R$ belongs to \textbf{H$(n-1,\ell)$} (c.f. \cite[3.3]{CD3} ).  Notice that when $R$ belongs to \textbf{H$(n-1,\ell)$}, 
$$\Ext_E^*(\widehat{R},k)\cong  k\oplus \shift^{-2} \frac{\A}{(\chi_n)}(\chi_j\alpha_i-\chi_i\alpha_j)_{{1\leq i<j\leq n-1}} \oplus\shift^{-4} \frac{\A^{m-n+1}}{(\chi_{n})}.$$ In particular, $\V(R)=\VV(\chi_n)$ is a hyperplane in $\mathbb{P}_k^{n-1}$. 

For $R$ belonging to {\bf G}$(r)$, {\bf TE}, {\bf B}, or \textbf{H$(p,q)$} for $p\neq n-1$ or $q\neq \ell$,  $\Ext_E^*(\widehat{R},k)$ contains a shift of an $\A$-free summand (see Theorem \ref{ethm}). Hence, $\V(R)=\mathbb{P}_k^{n-1}$ as claimed. 
\end{proof}

\begin{proposition}\label{p:ed}
With the setup in Notation \ref{n4},  if $Q\to \widehat{R}$ admits an embedded deformation then $\V(R)$ is contained in a hyperplane of $\mathbb{P}_k^{n-1}$. 
\end{proposition}
\begin{proof}
By Proposition \ref{compare}, up to a linear change of coordinates of $\mathbb{P}_k^{n-1}$, we can assume that $f_n$ is regular on $R':=Q/(f_1,\ldots, f_{n-1})$. Since $\pd_Q R'<\infty$ and $$f:=f_n+\sum_{i<n}a_if_i$$ is a $Q$-regular elements that is also $R'$-regular, it follows that $\pd_{Q/f} R'/fR'<\infty$. Furthermore, $$R'/fR'=Q/(\f)=\widehat{R}$$ and hence, $\pd_{Q/f}\widehat{R}<\infty$. Thus, by Theorem \ref{unstable}, we conclude that $$\V(R)\subseteq \VV(\chi_n).\qedhere$$
\end{proof}

As  indicated in Theorem \ref{clsv} and Proposition \ref{p:ed}, embedded deformations put a restriction on $\V(R)$. More generally, one says that $Q\to \widehat{R}$ admits an embedded deformation of codimension $c$ provided that there exists $g_1,\ldots, g_n\in \n^2$ with $g_{n-c+1},\ldots, g_n$ being a regular sequence on $Q/(g_1,\ldots, g_{n-c})$ and $$\widehat{R}\cong Q/(g_1,\ldots, g_n).$$ The author is curious as to whether $\V(R)$ can, in general, detect embedded deformations of arbitrary codimension; Theorem \ref{clsv} and Proposition \ref{p:ed}, as well as various specific examples, offer partial evidence for this. 

\begin{quest}Does the following  hold for a local ring $(R,\m,k)$?:
\begin{quote}
 If $Q\to\widehat{R}$ is a minimal  Cohen presentation of $R$, then $Q\to \widehat{R}$ admits an embedded deformation of codimension $c$ if and only if $\V(R)$ is contained in a hyperplane of codimension $c$ of $\PS_k^{n-1}$. \end{quote}
\end{quest}


\subsection{Solution to a Question of D. Jorgensen}\label{Djor}

This subsection is devoted to answering the following question of D. Jorgensen.
\begin{quest}\cite{Jor}\label{Question} Let $Q$ be a regular local ring with an algebraically closed residue field $k$. Set $R:=Q/(\f)$ where $\f\con \n^2$ minimally generates $(\f)$.
\begin{enumerate}
\item If $\Var(Q,\f;M,N)=\emptyset$ for some finitely generated $R$-modules $M$ and $N$, does  $\Ext_R^{\gg 0}(M,N)=0$?
\item If $\Var(Q,\f;M,N)=\emptyset$ for some finitely generated $R$-modules $M$ and $N$, is $R$ a complete intersection?
\end{enumerate}
\end{quest}

\begin{remark}
When $N$ is held fixed as $k$, then by Theorem \ref{thmls}(3) and \ref{relateJ} the answers to Questions \ref{Question}(1) and \ref{Question}(2) are ``yes."  However,  in general, the answers to  these questions are both  ``no" by Example \ref{counter}.\end{remark}

\begin{example}\label{counter}
Let $Q=k\lb x,y,z\rb$ where $k$ is algebraically closed. Set $\f=xy,yz$, $E:=\Kos^Q(\f)$, $\A:=k[\chi_1,\chi_2],$ and  $R:=Q/(\f)$. 

For each $p,q\in k$, define $$M_{p,q}:=R/(px+qz).$$ When $p\neq 0$, then $M_{p,q}=Q/(px+qz,yz)$. A minimal $Q$-free resolution of $M_{p,q}$ is given by 
$$F=0\to Qc\xrightarrow{\begin{pmatrix} -yz \\ px+qz \end{pmatrix} }Qb_1\oplus Qb_2\xrightarrow{ \begin{pmatrix} px + qz & yz \end{pmatrix}}  Qa\to 0$$ and since $\f M_{p,q}=0$ the DG $E$-module structure on $F$ is given by  
 $$(\lambda_1)_0=\begin{pmatrix} y/p \\ -q/p\end{pmatrix},  \ (\lambda_1)_1=\begin{pmatrix} q/p & y/p\end{pmatrix}\text{ and }(\lambda_2)_0=\begin{pmatrix} 0\\ 1\end{pmatrix}, \ (\lambda_2)_1=\begin{pmatrix} -1& 0\end{pmatrix}$$  where $\lambda_i$ denotes left multiplication by $\xi_i$ on $F$. That is,  $F$ is a Koszul resolution of $M_{p,q}$ with $\lambda_1$ and $\lambda_2$ prescribed above.

By Remark \ref{complex},   $\CC_E(F,k)$ is the following complex of graded $\A$-modules $$0\to \shift^{-4}\A\gamma \xrightarrow{\begin{pmatrix}  \frac{q}{p} \chi_1-\chi_2  \\ 0\end{pmatrix}}\shift^{-2} \A\beta_1\oplus \A\beta_2 \xrightarrow{\begin{pmatrix}0 & \chi_2-\frac{q}{p}\chi_1 \end{pmatrix}} \A\alpha\to 0.$$
Hence, $$\Ext_{E}^*(M_{p,q},k)= k[\chi_1,\chi_2]/(q\chi_1-p\chi_2)\oplus \shift^{-2}k[\chi_1,\chi_2]/(q\chi_1-p\chi_2).$$
A similar argument, holds for $q\neq 0$.

Therefore, for each point $(p,q)\in \mathbb{A}_k^2\backslash\{(0,0)\}$, we have that $$\VV_E(M_{p,q},k)=\gsupp{\A} (\A/(q\chi_1-p\chi_2)).$$  In particular, for each point $(p,q)\in \mathbb{A}_k^2$  $$\Var(Q,\f,M_{p,q},k)=\{(a,b)\in \mathbb{A}_k^2: qa=pb\}$$ is the line through $(0,0)$ and $(p,q)$ in $\mathbb{A}_k^2$.  Also,  by \ref{relateJ}, $$\Var(Q,\f;M_{p,q},k)\cap \Var(Q,\f;M_{s,t},k)=\Var(Q,\f,M_{p,q}, M_{s,t}).$$
Therefore, for two points $\bm{a}$ and $\bm{b}$ in $\mathbb{A}_k^2\backslash\{(0,0)\}$ such that  $\bm{a}\neq \lambda \bm{b}$ for any $\lambda$, we have that $$\Var(Q,\f;M_{\bm{a}}, M_{\bm{b}})=\emptyset.$$  Therefore, this answers Question \ref{Question}(2) in the negative. 
Finally, a direct calculation shows that  $\Ext_R^*(M_{1,0}, M_{0,1})$ is unbounded and hence, the answer to Question \ref{Question}(1) is also ``no," in general. 
\end{example}

\begin{remark} When $R$ is a complete intersection, every closed subset of $\PS_k^{n-1}$ is realizable as $\V_R(M)$ for some finitely generated $R$-module $M$ (see \cite{AI2,BerghCI,BW}).
If $R$ is not a complete intersection, then $\V_R(M)\neq \emptyset$ whenever $M$ is a  nonzero finitely generated $R$-module  (see Theorem \ref{thmls}(3)). In Example \ref{counter}, we demonstrate that nearly every non-empty closed subset of $\mathbb{P}_k^1$ is realizable as $\V_R(M)$ for some finitely generated $R$-module $M$. However, in general, it is not known which closed subsets can be attained as $\V_R(M)$ for some finitely generated $R$-module $M$ (or object $M$ of $\D^f(R)$). In \cite[3.3.4]{Pol}, the author shows every hyperplane of $\mathbb{P}_k^{n-1}$ is realizable as a complex whose homology is just two copies of $k$ in specified degrees. It is not known if one can use \emph{finitely generated modules} to obtain every hyperplane of $\mathbb{P}_k^{n-1}$.\end{remark}

\begin{Problem}\label{problem}
Determine what subsets of $\mathbb{P}_k^{n-1}$ are realizable as $\V_R(M)$ for a finitely generated $R$-module $M$, or, more generally, for $M$ in $\D^f(R)$. 
\end{Problem}


\bibliographystyle{abbrv}
\bibliography{refs}

\begin{thebibliography}{10}

\bibitem{VPD}
L.~L. Avramov.
\newblock Modules of finite virtual projective dimension.
\newblock {\em Inventiones mathematicae}, 96(1):71--101, 1989.

\bibitem{IFR}
L.~L. Avramov.
\newblock Infinite free resolutions.
\newblock In {\em Six lectures on commutative algebra}, pages 1--118. Springer,
  1998.

\bibitem{CD3}
L.~L. Avramov.
\newblock A cohomological study of local rings of embedding codepth 3.
\newblock {\em J. Pure Appl. Algebra}, 216(11):2489--2506, 2012.

\bibitem{CD2}
L.~L. Avramov and R.-O. Buchweitz.
\newblock Homological algebra modulo a regular sequence with special attention
  to codimension two.
\newblock {\em Journal of Algebra}, 230(1):24--67, 2000.

\bibitem{SV}
L.~L. Avramov and R.-O. Buchweitz.
\newblock Support varieties and cohomology over complete intersections.
\newblock {\em Inventiones mathematicae}, 142(2):285--318, 2000.

\bibitem{HPC}
L.~L. Avramov, R.-O. Buchweitz, S.~B. Iyengar, and C.~Miller.
\newblock Homology of perfect complexes.
\newblock {\em Advances in Mathematics}, 223(5):1731--1781, 2010.

\bibitem{AGP}
L.~L. Avramov, V.~N. Gasharov, and I.~V. Peeva.
\newblock Complete intersection dimension.
\newblock {\em Publications Math{\'e}matiques de l'Institut des Hautes
  {\'E}tudes Scientifiques}, 86(1):67--114, 1997.

\bibitem{AG}
L.~L. Avramov and D.~R. Grayson.
\newblock Resolutions and cohomology over complete intersections.
\newblock In {\em Computations in algebraic geometry with {M}acaulay 2},
  volume~8 of {\em Algorithms Comput. Math.}, pages 131--178. Springer, Berlin,
  2002.

\bibitem{AI}
L.~L. Avramov and S.~B. Iyengar.
\newblock Constructing modules with prescribed cohomological support.
\newblock {\em Illinois Journal of Mathematics}, 51(1):1--20, 2007.

\bibitem{AI3}
L.~L. Avramov and S.~B. Iyengar.
\newblock Cohomology over complete intersections via exterior algebras.
\newblock In {\em Triangulated categories}, volume 375 of {\em London Math.
  Soc. Lecture Note Ser.}, pages 52--75. Cambridge Univ. Press, Cambridge,
  2010.

\bibitem{AI2}
L.~L. Avramov and S.~B. Iyengar.
\newblock Restricting homology to hypersurfaces.
\newblock In {\em Geometric and topological aspects of the representation
  theory of finite groups}, volume 242 of {\em Springer Proc. Math. Stat.},
  pages 1--23. Springer, Cham, 2018.

\bibitem{AKM}
L.~L. Avramov, A.~R. Kustin, and M.~Miller.
\newblock Poincar\'{e} series of modules over local rings of small embedding
  codepth or small linking number.
\newblock {\em J. Algebra}, 118(1):162--204, 1988.

\bibitem{AS}
L.~L. Avramov and L.-C. Sun.
\newblock Cohomology operators defined by a deformation.
\newblock {\em Journal of Algebra}, 204(2):684--710, 1998.

\bibitem{BIK}
D.~Benson, S.~B. Iyengar, and H.~Krause.
\newblock Local cohomology and support for triangulated categories.
\newblock {\em Ann. Sci. {\'E}c. Norm. Sup{\'e}r.(4)}, 41(4):573--619, 2008.

\bibitem{BerghCI}
P.~A. Bergh.
\newblock On support varieties for modules over complete intersections.
\newblock {\em Proc. Amer. Math. Soc.}, 135(12):3795--3803, 2007.

\bibitem{BEJ}
B.~Briggs, E.~Grifo, and J.~Pollitz.
\newblock Constructing non-proxy small test modules for the complete
  intersection property.
\newblock {\em arXiv e-prints}, pages 1--14, 2020.
\newblock arXiv:2009.11800v2.

\bibitem{BH}
W.~Bruns and H.~J. Herzog.
\newblock {\em Cohen-macaulay rings}.
\newblock Cambridge University Press, 1998.

\bibitem{BE}
D.~A. Buchsbaum and D.~Eisenbud.
\newblock Algebra structures for finite free resolutions, and some structure
  theorems for ideals of codimension {$3$}.
\newblock {\em Amer. J. Math.}, 99(3):447--485, 1977.

\bibitem{Buch}
R.-O. Buchweitz.
\newblock Maximal cohen-macaulay modules and tate-cohomology over gorenstein
  rings.

\bibitem{BW}
J.~Burke and M.~E. Walker.
\newblock Matrix factorizations in higher codimension.
\newblock {\em Trans. Amer. Math. Soc.}, 367(5):3323--3370, 2015.

\bibitem{CI}
J.~F. Carlson and S.~B. Iyengar.
\newblock Thick subcategories of the bounded derived category of a finite
  group.
\newblock {\em Trans. Amer. Math. Soc.}, 367(4):2703--2717, 2015.

\bibitem{Lars}
L.~W. Christensen.
\newblock {\em Gorenstein dimensions}, volume 1747 of {\em Lecture Notes in
  Mathematics}.
\newblock Springer-Verlag, Berlin, 2000.

\bibitem{DGI}
W.~Dwyer, J.~P. Greenlees, and S.~B. Iyengar.
\newblock Finiteness in derived categories of local rings.
\newblock {\em Commentarii mathematici helvetici}, 81(2):383--432, 2006.

\bibitem{Eis}
D.~Eisenbud.
\newblock Homological algebra on a complete intersection, with an application
  to group representations.
\newblock {\em Transactions of the American Mathematical Society},
  260(1):35--64, 1980.

\bibitem{FHT}
Y.~F{\'e}lix, S.~Halperin, and J.-C. Thomas.
\newblock {\em Rational homotopy theory}, volume 205.
\newblock Springer Science \& Business Media, 2012.

\bibitem{FM}
L.~Ferraro and W.~F. Moore.
\newblock Differential graded algebra over quotients of skew polynomial rings
  by normal elements.
\newblock {\em Trans. Amer. Math. Soc.}, 373(11):7755--7784, 2020.

\bibitem{FMP}
L.~Ferraro, W.~F. Moore, and J.~Pollitz.
\newblock Support varieties over skew complete intersections via derived
  braided hochschild cohomology.
\newblock {\em arXiv preprint arXiv:2101.12287}, 2021.

\bibitem{FIJ}
A.~Frankild, S.~Iyengar, and P.~J\o~rgensen.
\newblock Dualizing differential graded modules and {G}orenstein differential
  graded algebras.
\newblock {\em J. London Math. Soc. (2)}, 68(2):288--306, 2003.

\bibitem{M2}
D.~R. Grayson and M.~E. Stillman.
\newblock Macaulay2, a software system for research in algebraic geometry.
\newblock Available at \url{http://www.math.uiuc.edu/Macaulay2/}.

\bibitem{G}
T.~H. Gulliksen.
\newblock A change of ring theorem with applications to poincar{\'e} series and
  intersection multiplicity.
\newblock {\em Mathematica Scandinavica}, 34(2):167--183, 1974.

\bibitem{GL}
T.~H. Gulliksen and G.~Levin.
\newblock {\em Homology of local rings}.
\newblock Number 20-21. Kingston, Ont.: Queen's University, 1969.

\bibitem{HJ}
C.~Huneke and D.~A. Jorgensen.
\newblock Symmetry in the vanishing of {E}xt over {G}orenstein rings.
\newblock {\em Math. Scand.}, 93(2):161--184, 2003.

\bibitem{Jor}
D.~A. Jorgensen.
\newblock Support sets of pairs of modules.
\newblock {\em Pacific journal of mathematics}, 207(2):393--409, 2002.

\bibitem{PJor}
P.~J{\o}rgensen.
\newblock Amplitude inequalities for differential graded modules.
\newblock In {\em Forum Mathematicum}, volume~22, pages 941--948, 2010.

\bibitem{Mehta}
V.~B. Mehta.
\newblock Endomorphisms of complexes and modules over golod rings.
\newblock 1976.

\bibitem{Pol}
J.~Pollitz.
\newblock The derived category of a locally complete intersection ring.
\newblock {\em Adv. Math.}, 354:106752, 18, 2019.

\bibitem{Stevenson}
G.~Stevenson.
\newblock Subcategories of singularity categories via tensor actions.
\newblock {\em Compos. Math.}, 150(2):229--272, 2014.

\end{thebibliography}

\end{document}